\documentclass[11pt,a4paper]{amsart}
\usepackage[english,frenchb]{babel}
\usepackage[utf8]{inputenc}  
\usepackage[T1]{fontenc}  
\usepackage{amscd,amssymb,amsmath,latexsym,url,mathrsfs,upgreek,graphicx}
\usepackage[all]{xy}
\usepackage[section]{algorithm}
\usepackage{algorithmic}
\usepackage{hyperref} \hypersetup{breaklinks=true}

\def\ov#1{{\overline{#1}}}

\def\wt#1{{\widetilde{#1}}}

\newcommand{\codim}{{\operatorname{codim}}}
\newcommand{\rank}{{\operatorname{rank}}}
\newcommand{\h}{{\operatorname{h}}}
\newcommand{\ord}{{\operatorname{ord}}}

\newcommand{\Aut}{{\operatorname{Aut}}}
\newcommand{\Hom}{{\operatorname{Hom}}}
\newcommand{\supp}{{\operatorname{supp}}}
\newcommand{\irr}{{\operatorname{irr}}}

\newcommand{\Qb}{\overline{\mathbb Q}}
\newcommand{\Gm}{{\mathbb G}_{\rm m}}

\newcommand{\id}{{\operatorname{id}}}
\newcommand{\im}{{\operatorname{im}}}
\newcommand{\MM}{{\rm M}}

\newcommand{\size}{{\rm size}}
\newcommand{\trans}{{\rm t}}
\newcommand{\exc}{{\rm exc}}
\newcommand{\OmegaW}{{\bigcup\Omega_{W}}}

\newcommand{\sat}{{\rm sat}}
\newcommand{\constant}{{c}}

\newcommand{\A}{{\mathbb{A}}}

\newcommand{\N}{{\mathbb{N}}} \renewcommand{\P}{{\mathbb{P}}}
\newcommand{\Q}{{\mathbb{Q}}} \newcommand{\R}{{\mathbb{R}}}
 \newcommand{\Z}{{\mathbb{Z}}}

\newcommand{\bfa}{{\boldsymbol{a}}}
\newcommand{\bfb}{{\boldsymbol{b}}}
\newcommand{\bfc}{{\boldsymbol{c}}}

\newcommand{\bfe}{{\boldsymbol{e}}}
\newcommand{\bff}{{\boldsymbol{f}}}

\newcommand{\bfx}{{\boldsymbol{x}}}
\newcommand{\bfy}{{\boldsymbol{y}}}

\newcommand{\bfE}{{\boldsymbol{E}}}
\newcommand{\bfF}{{\boldsymbol{F}}}
\newcommand{\bfG}{{\boldsymbol{G}}}
\newcommand{\bfP}{{\boldsymbol{P}}}
\newcommand{\bfQ}{{\boldsymbol{Q}}}

\newcommand{\bfeta}{{\boldsymbol{\eta}}}
\newcommand{\bfzeta}{{\boldsymbol{\zeta}}}
\newcommand{\bflambda}{{\boldsymbol{\lambda}}}

\newcommand{\bfxi}{{\boldsymbol{\xi}}}

\newcommand{\bfmu}{{\boldsymbol{\mu}}} 

\newcommand{\bfzero}{{\boldsymbol{0}}}

\newcounter{thm}
\numberwithin{equation}{section}
\numberwithin{thm}{section}

\theoremstyle{definition}
\newtheorem{definition}[thm]{Definition}
\newtheorem{notation}[thm]{Notation}
\newtheorem{remark}[thm]{Remark}
\newtheorem{example}[thm]{Example}

\theoremstyle{plain}
\newtheorem{lemma}[thm]{Lemma}
\newtheorem{proposition}[thm]{Proposition}
\newtheorem{theorem}[thm]{Theorem}
\newtheorem{corollary}[thm]{Corollary}
\newtheorem{conjecture}[thm]{Conjecture}
\newtheorem{problem}[thm]{Problem}
\newtheorem{prop-def}[thm]{Proposition-Definition}

\algsetup{linenodelimiter=.}

\begin{document}

\selectlanguage{english}

\title[Systems of sparse polynomial equations]{Overdetermined systems of sparse polynomial equations}

\author[Amoroso]{Francesco Amoroso}
\address{Laboratoire de math\'ematiques Nicolas Oresme, CNRS 
UMR 6139, Universit\'e de Caen, BP 5186, 14032 Caen Cedex, France}
\email{francesco.amoroso@unicaen.fr}
\urladdr{\url{http://www.math.unicaen.fr/~amoroso/}}

\author[Leroux]{Louis Leroux}
\address{Laboratoire de math\'ematiques Nicolas Oresme, CNRS 
UMR 6139, Universit\'e de Caen,  BP 5186, 14032 Caen Cedex, France}
\email{louis.leroux@ac-caen.fr}
\urladdr{\url{http://www.math.unicaen.fr/~lleroux/}}

\author[Sombra]{Mart{\'\i}n~Sombra}
\address{ICREA \&
Departament d'{\`A}lgebra i Geometria, Universitat de Barcelona.
Gran Via~585, 08007 Barcelona, Spain}
\email{sombra@ub.edu}
\urladdr{\url{http://atlas.mat.ub.es/personals/sombra/}}

\date{\today} 

\thanks{Amoroso and Leroux were partially supported by the CNRS
  research project PICS \og Properties of the heights of arithmetic
  varieties\fg{}. Amoroso was also partially supported by the ANR
  research project \og Hauteurs, modularit\'e,
  transcendance\fg{}. Sombra was partially supported by the MICINN
  research projects MTM2009-14163-C02-01 and MTM2012-38122-C03-02.}

\begin{abstract}
  We show that, for a system of univariate polynomials given in sparse
  encoding, we can compute a single polynomial defining the same zero
  set, in time quasi-linear in the logarithm of the degree. In
  particular, it is possible to determine whether such a system of
  polynomials does have a zero in time quasi-linear in the logarithm
  of the degree. The underlying algorithm relies on a result of
  Bombieri and Zannier on multiplicatively dependent points in
  subvarieties of an algebraic torus.

  We also present the following conditional partial extension to the
  higher dimensional setting. Assume that the effective Zilber
  conjecture holds. Then, for a system of multivariate
  polynomials given in sparse encoding, we can compute a finite
  collection of complete intersections outside hypersurfaces that
  defines the same zero set, in time quasi-linear in the logarithm of
  the degree.
\end{abstract}
\maketitle

\section{Introduction}

A system of polynomial equations
\begin{equation} \label{eq:18}
  f_{1}=\dots=f_{s}=0
\end{equation}
is \og overdetermined\fg{} if the number of equations exceeds the codimension of
its zero set. Our aim is to give algorithms for reducing those systems
of equations to a finite number of \og well-determined\fg{} systems.
We focus on the case when the input polynomials are sparse in the
sense that they have  high degree but relatively few nonzero
terms and small coefficients, and we want our algorithms to be as
efficient as possible in that situation.

\bigskip For univariate polynomials, the reduction of an
overdetermined system as in \eqref{eq:18} with $s\ge 2$, might be done by computing
the greatest common divisor of the $f_{i}$'s. However, this strategy
does not work in our situation because the gcd of a family of sparse
polynomials is not necessarily sparse, as shown by the following
example due to Schinzel~\cite{Schinzel:gcdtup}: if $a,b\ge 1$ are
coprime, then
\begin{displaymath}
  \gcd( x^{ab}-1, (x^{a}-1)(x^{b}-1)) =  \frac{(x^{a}-1)(x^{b}-1)}{x-1}.
\end{displaymath}
This polynomial has $2\min(a,b)$ nonzero terms. Hence, for $a,b\gg0$, both
$x^{ab}-1$ and $(x^{a}-1)(x^{b}-1)$ are sparse, but their gcd is
not. 

This example suggests that one should avoid polynomials vanishing at
roots of unity. Indeed, Filaseta, Granville and Schinzel have shown
that, if $f,g\in \Z[x]$ are given in sparse encoding and either
$f$ or $g$ do not vanish at any root of unity, then $\gcd(f,g)$ can be
computed with $\wt O(\log (d))$~{ops}
\cite{FilasetaGranvilleSchinzel:igcdasp}. Here, \emph{ops} is an
abbreviation of \og bit operations\fg{}, and the \og soft~O\fg{}
notation indicates a bound with an extra factor of type
$(\log\log(d))^{O(1)}$ and implicit constants depending on the  size
of the coefficients and number of nonzero terms of $f$ and
$g$.  Their algorithm relies heavily on a theorem of Bombieri and
Zannier on the intersection of a subvariety of the algebraic torus
with subtori of dimension~$1$~\cite[Appendix]{Schinzel:psrr}, see
also \S~\ref{sec:unlik-inters-tori}.

Our first result is an extension of the algorithm of Filaseta,
Granville and Schinzel, allowing the case when both $f$ and $g$ vanish
at roots of unity.  From now on, we fix a number field $K\subset
\Qb$. Given \begin{math} f_{1}, \dots, f_{s}\in K
  [x_{1},\dots,x_{n}]
\end{math}, we denote by $V(f_{1},\dots, f_{s})$ their set of common
zeros in the affine space $\A^{n}=\Qb^{n}$. Polynomials are given in
sparse encoding. Recall that the height of a polynomial is a measure
of the bit size of its coefficients, see \S~\ref{sec:integ-laur-polyn}
for details.  The following statement makes the output and the
complexity of our algorithm precise.

\begin{theorem} \label{thm:1} There is an algorithm that, given
  $f,g\in K[x]$, computes $p_{1},p_{2}\in K[x]$ such that
  \begin{displaymath}
    p_{1}|\gcd(f,g), \quad V(p_{1})\setminus \upmu_{\infty}=
    V(\gcd(f,g))\setminus \upmu_{\infty}, \quad  \text{ and }\quad  
V(p_{2})=V(\gcd(f,g)) \cap \upmu_{\infty}, 
  \end{displaymath}
where $\upmu_{\infty}$ denotes the subgroup of $\Qb^{\times}$ of roots
of unity. 

If both $f$ and $g$ have degree bounded by $ d$ and height and number of
nonzero coefficients bounded by $\constant$, this computation is done with
$\wt O(\log (d))$ ops, where the implicit constants depend only on $K$
and $\constant$.
\end{theorem}


This result is a simplified version of Theorem \ref{gcd}, which holds
for families of univariate polynomials and gives more information about
the output polynomials. The underlying procedure is given by
Algorithms \ref{alg:27} and \ref{alg:28}.
A preliminary version appears in the second author's
Ph.D. thesis~\cite{Leroux:apl}.

In the notation of Theorem \ref{thm:1},  we have that
  \begin{equation}\label{eq:1}
    V(p_{1}p_{2})=V(f,g).
  \end{equation}
  Hence, given two univariate polynomials with bounded height and
  number of nonzero coefficients, we can compute a polynomial which
  has the same zero set as their gcd, with complexity quasi-linear in the
  logarithm of the degree. In particular, we deduce the following
  corollary. 

\begin{corollary} \label{cor:1} Let $f,g\in K[x]$ be
  polynomials of degree~$\leq d$ and  height and number of nonzero coefficients
  bounded by a constant $\constant$.   We can decide if
  \begin{displaymath}
\gcd(f,g)=1
  \end{displaymath}
  with $\wt O(\log (d))$ ops, where the implicit constants depend only
  on $K$ and $\constant$.
\end{corollary}

A classical result of Plaisted says that computing the degree of the
gcd of two univariate polynomials given in sparse encoding is an
NP-hard problem~\cite{Plaisted:scppr}. Using Plaisted's techniques in
\emph{loc. cit.}, it can be shown that already deciding if the degree
of the gcd is zero, is an NP-hard problem. Hence, if Cook's conjecture
${\rm P} \ne {\rm NP}$ holds, it is not possible to decide if the
degree of the gcd is zero with a complexity
which is polynomial in the height, number of nonzero terms, and
logarithm of the degree of the input polynomials. In contrast to this,
Corollary \ref{cor:1} shows that this problem can be solved with a
complexity which is quasi-linear in the logarithm of the degree
although, \emph{a priori}, not polynomial in the height and number of
nonzero terms.

\bigskip

Our algorithms for the multivariate case rely on an effective
version of the Zilber conjecture generalizing the quoted theorem of
Bombieri and Zannier. For $N\ge0$, we denote by
$\Gm^N=(\Qb^{\times})^{N}$ the (split) algebraic torus over $\Qb$ of
dimension~$N$. Recall that a {torsion coset} of $\Gm^{N}$ is a
connected component of an algebraic subgroup or, equivalently, a
translate of a subtorus by a torsion point. The effective Zilber
conjecture can then be stated as follows:

\medskip 
\begin{quote} {\it 
Let
  $W$ be an irreducible subvariety of $\Gm^N$. There exists a finite
  and effectively calculable collection $\Omega$ of torsion cosets of
  $\Gm^N$ of codimension 1 such
  that, if  $B$ is a torsion coset of
  $\Gm^N$ and $C$ an irreducible component of $B\cap W$ such that
$$
\dim(C) > \dim(B) - \codim(W), 
$$
then there exists $T\in \Omega$ such that $C\subset T$.}
\end{quote} 

Zilber proposed this conjecture (under the equivalent formulation that
we recall in Conjecture \ref{Zilber_Conj1}) in connection with the so-called
\og uniform Schanuel conjecture\fg{} and motivated by problems from
model theory \cite{Zilber:esesc}.  It is still unproven, but several
interesting cases are already known.  When we restrict to $\dim(B)=0$,
the statement is equivalent to the toric case of the Manin-Mumford
conjecture. This is a well-known theorem of Laurent~\cite{Laurent:ede}
and an effective proof of it can be found in Schmidt's paper
\cite{Schmidt:hps} or, more explicitly, in the second author's paper
\cite{Leroux:ctpvdlp}. The result by Bombieri and Zannier solves the
case when we restrict to $\dim (B)=1$. The case when $W$ is a curve
was proved by Maurin~\cite{Maurin:caem}, building on previous work by
Bombieri, Masser and Zannier
\cite{BombieriMasserZannier:icasmg}. Moreover, the closely related \og
bounded height conjecture\fg{} has been proved by Habegger in the
general case~\cite{Habegger:bhc}.

The Zilber conjecture plays a central role in the study of \og
unlikely intersections\fg{} and has many applications in number
theory, see for instance the survey \cite{CL:rdie} and the book
\cite{Zannier:spuiag} for accounts of this very active area of
research.

\bigskip A well-determined system of polynomial equations is, by
definition, a complete intersection.  The solution set of a system of
multivariate polynomial equations cannot always be redefined by a
single complete intersection, since it might have components of
different codimensions and, moreover, these components might not be
complete intersections either. Instead, this solution set can be
described as a finite union of complete intersections on open subsets
(Proposition-Definition~\ref{def:2}). Such a decomposition can be
understood as a sort of generalization to the multivariate setting of
the polynomial~$p_{1}p_{2}$ in~\eqref{eq:1} from the univariate case.

The main result of this paper is an algorithm giving a conditional
partial computation of this decomposition for an arbitrary system of
multivariate polynomials (Algorithm \ref{alg:5}). The following
statement makes the size of its output and its complexity precise.

\begin{theorem} \label{main-s} Assume that the effective Zilber
  conjecture holds.  There is an algorithm that, given  $f_{1},\dots,
  f_{s}\in K[x_{1},\dots, x_{n}]$, computes a finite collection
  $\Gamma$ whose elements are sequences $ (p_{1},\ldots,p_{r},q) $ of
  polynomials in $K(\omega)[x_{1},\dots, x_{n}]$ with $\omega$ a root
  of unity, such that either
  $\codim(V(p_{1},\ldots,p_{r})\backslash V(q))=r$ or
  $V(p_{1},\ldots,p_{r})\backslash V(q)=\emptyset$, and
  \begin{equation}\label{eq:8}
V(f_{1},\dots, f_{s})=\bigcup_{(p_{1},\ldots,p_{r},q) \in \Gamma} V(p_1,\ldots,p_r)\backslash V(q).
  \end{equation}

  If both $n$ and $s$ are bounded by a constant $\constant$ and each
  $f_{i}$ is of degree $\leq d$ and  height and number of nonzero
  terms bounded by $\constant$, then the cardinality of $\Gamma$ is
  bounded by $O(1)$, the order of $\omega$ is bounded by $O(1)$, the
  polynomials in $\Gamma$ have degree bounded by $d^{O(1)}$, height
  and number of nonzero coefficients bounded by $O(1)$, and the
  computation is done with $\wt O(\log (d))$ ops, where the implicit
  constants depend only on~$K$ and $\constant$.
\end{theorem}

A previous result in this direction, for systems of three polynomials
in two variables, appears in the second author's Ph.D. thesis
\cite{Leroux:apl}.

From~\eqref{eq:8}, one can derive a well-determined description of
$V(f_{1},\dots, f_{s})$ by throwing away the empty pieces. The bounds
in Theorem \ref{main-s} imply that, if the input polynomials are
sparse, then this decomposition into complete intersections outside
hypersurfaces is defined by polynomials that are also sparse. The
actual computation of such a decomposition from the output of
Algorithm \ref{alg:5} amounts to deciding when
$V(p_1,\ldots,p_r)\backslash V(q)=\emptyset$ for each $
(p_{1},\ldots,p_{r},q) \in \Gamma$. Unfortunately, it is not clear yet
how to perform this task with $\wt O(\log(d))$ ops, see
Problem~\ref{prob:2} below.


\bigskip The idea for the algorithms underlying Theorems \ref{thm:1}
and \ref{main-s} is inspired by the method in
\cite{FilasetaGranvilleSchinzel:igcdasp}. It can be explained as
follows. First, by decomposing the affine space into a disjoint union
of algebraic tori, the problem can be reduced to the analogous problem
on the open subset $\Gm^{n}\subset \A^{n}$.  Write
\begin{equation}\label{eq:24}
  f_i=\sum_{j=1}^N \alpha_{i,j}\bfx^{\bfa_{j}}\in K[x_{1},\dots, x_{n}], \quad
i=1,\dots, s,
\end{equation}
with $\alpha_{i,j}\in K$ and $\bfa_{j}\in \Z^{n}$. Then we consider
the homomorphism $\varphi\colon \Gm^{n}\rightarrow \Gm^{N}$ defined by
the exponents $\bfa_{j}$, $j=1,\dots, s$ and the linear forms
$\ell_{i}=\sum_{j=1}^N \alpha_{i,j}y_j$, $i=1,\dots,s$. These forms
define a linear subvariety $W$ of $\Gm^N$ such that
\begin{equation}\label{eq:25}
\varphi^{-1}(W)=  V(f_{1},\dots, f_{s})\cap \Gm^{n}.
\end{equation}
Roughly speaking, the next step consists of producing a stratification
of $W$ by successively intersecting  this subvariety with
torsion cosets of codimension 1 produced by the effective Zilber conjecture.
The output of the algorithm is obtained by considering the inverse
image under the homomorphism $\varphi$ of some of the pieces of this
stratification.


\bigskip
Several questions arose during our work, and we close this introduction by
pointing out two of them. Both problems have interesting algorithmic
consequences and seem to be related to the study of unlikely
intersections in algebraic tori.

\begin{problem} \label{prob:1} Give an algorithm for computing the
  degree of the gcd of two polynomials given in the sparse encoding,
  of degree $\le d$ and bounded height and number of coefficients,
  with $\wt O(\log (d))$~ops.
\end{problem}

An affirmative answer to Problem \ref{prob:1} would allow  to test
divisibility of sparse polynomials with complexity quasi-linear in the
logarithm of their degree.

\begin{problem} \label{prob:2}
  In the setting of Theorem~\ref{main-s}, modify
  the underlying algorithm to exclude the possibility that
  $V(p_{1},\ldots,p_{r})\backslash V(q)=\emptyset$. 
\end{problem}

An affirmative answer to Problem \ref{prob:2} would allow us to
compute the dimension of $V(f_{1},\dots, f_{s})$ with $\wt O(\log
(d))$ ops.  In particular, we could then determine whether the zero
set $V(f_{1},\dots, f_{s})$ is empty with $\wt O(\log (d))$ ops,
extending Corollary \ref{cor:1} to the multidimensional case.

\medskip \noindent {\bf Acknowledgments.}  We thank Ga\"el R\'emond
and Umberto Zannier for useful discussions about unlikely
intersections and the Zilber conjecture. We also thank the referees 
for their many remarks that helped us to improve our presentation.

Part of this work was done while the authors met at the Universitat de
Barcelona, the Université de Caen and the Centro di Ricerca Matematica
Ennio de Giorgi (Pisa). We thank these institutions for their hospitality. 

\section{Notation and auxiliary results} \label{sec:notat-auxil-results-1}

We fix a number field $K\subset \Qb$. Bold letters denote finite sets
or sequences of objects, where the type and number should be clear
from the context: for instance, $\bfx$ might denote the group of
variables $\{x_1,\dots,x_n\}$, so that $K[\bfx^{\pm1}]$ denotes the
ring of Laurent polynomials $K[x_1^{\pm1},\dots,x_n^{\pm1}]$.

Given functions $f,g\colon \N\to \R_{>0}$, the Landau symbols $f=O(g)$
and $f=\wt O(g)$ respectively mean that there are positive constants
$c_{1},c_{2}\ge 0$
such that, for all $m\in \N$, 
\begin{displaymath}
f(m)\le c_{1}\, g(m), \quad f(m)\le c_{1}\, g(m) \max(1,\log(g(m)))^{c_{2}}.
\end{displaymath}
If we want to emphasize the dependence of the constants $c_{1}$ and
$c_{2}$ on parameters, say $N$ and $h$, we will write $f=O_{N,h}(g)$
and $f=\wt O_{N,h}(g)$, respectively.  When these parameters are said
to be \og bounded\fg{}, we omit them from the notation as we do, for
instance, in Lemma \ref{torsion-coset}.

\subsection{Integers and Laurent
  polynomials} \label{sec:integ-laur-polyn}

We denote by $\N$, $\Z$ and $\Q$ the monoid of natural numbers
with~0, the ring of integers and the field of rational numbers,
respectively.  At the computational level, integers are represented in
bit encoding and rational numbers are represented as quotients
of integers.  The complexity of an algorithm will be measured in bit
operations (\emph{ops}).

A \emph{multiplication time function} is a function 
\begin{displaymath}
\MM\colon\N\to\N  
\end{displaymath}
such that integers of bit length $\le k$ can be multiplied using at
most $\MM(k)$ ops. We also assume that, for $k, l\in \N$, this function
verifies $ \MM(kl)\le k^{2}\MM(l)$ and, if $k\ge l$, it also
verifies ${\MM(k)}/{k}\ge {\MM(l)}/{l}$.

Such a function dominates the complexity of many of the basic
computations on $\Z$. In particular, for integers of bit length $\le
k$, division with remainder can be done with $O(\MM(k))$ ops, and their
gcd can be computed with $O(\MM(k)\log(k))$ ops
\cite{vzGG:mca}.  By the Sch\"onhage-Strassen algorithm \cite[Theorem
8.24]{vzGG:mca}, we can take
\begin{displaymath}
 \MM(k)= O(k\log(k)\log(\log(k)))= \wt O (k).  
\end{displaymath}

The number field $K$ is represented by a monic irreducible polynomial
$h\in \Q[z]$ such that $K\simeq \Q[z]/(h)$. The arithmetic operations
of $K$ (sum, difference, multiplication and division by a nonzero
element) can be computed in terms of this representation.  We will not
be concerned by their complexity, since it will be absorbed by the
constants in our bounds.

For $l\ge 1$, we denote by $\upmu_{l}$ the subgroup of $\Qb^{\times}$
of roots of unity of order dividing $l$.  We also set $\upmu_{\infty}$
for the subgroup of $\Qb^{\times}$ of all roots of unity. Hence,
\begin{displaymath}
  \upmu_{\infty}=\bigcup_{l\ge1} \upmu_{l}.
\end{displaymath}
For $N\ge0$, an $N$-tuple 
  of roots of unity $\bfeta\in
\upmu_{l}^{N}$ is represented as $\bfeta= (\omega^{i_{1}}, \dots,
\omega^{i_{N}})$ with $\omega$  a primitive $l$-th root of unity
and $0\le i_{j}\le l-1$. A representation of the
finite extension $K(\bfeta)$ can be computed in terms of $\bfeta$ and
a representation of $K$. Again, the complexity of computing this
representation will play no role in our results.

Laurent polynomials will be represented in sparse encoding: a
Laurent polynomial $f\in K[x_1^{\pm1},\ldots,x_n^{\pm1}]$ will be
given by a sequence of pairs $(\bfa_{j},\alpha_{j})\in \Z^{n}\times K$, $ j=1,\dots, N$,
such that
\begin{displaymath}
 f= \sum_{j=1}^{N}\alpha_{j}\bfx^{\bfa_{j}}.  
\end{displaymath}
We assume that $\bfa_{j}\ne \bfa_{k}$ for $j\ne k$.

For a vector $\bfa=(a_{1},\dots, a_{n})\in \Z^{N}$, we denote its
$\ell^{1}$-norm by
$$|\bfa|=|a_{1}|+\dots+|a_{n}|.$$ 
We respectively define the \emph{support} and the \emph{degree} of $f$ as
\begin{displaymath}
\supp(f)=\{\bfa_{j}\mid \alpha_{j}\ne
0\}, \quad \deg(f)=
\max_{\bfa_{j}\in \supp(f)}|\bfa_{j}|.
\end{displaymath}
When $f$ is a polynomial, this notion of degree coincides with the
usual one.  

We define the \emph{height} of $f$, denoted $\h(f)$, as
the Weil height of the projective point
$(1:\alpha_{1}:\cdots:\alpha_{N})\in \P^{N}$, see for instance
\cite[Chapter 2]{BombieriGubler:heights} or \cite[Chapter
3]{Zannier:lda} for details.  In the particular case when $f\in
\Z[x_1^{\pm1},\ldots,x_n^{\pm1}]$,
\begin{displaymath}
  \h(f)= \max_{j}  \log|\alpha_{j}|.
\end{displaymath}

Ideals of $K[x_1^{\pm1},\ldots,x_n^{\pm1}]$ are represented by
finite families of generators. 

\subsection{Subvarieties and locally closed subsets}\label{sec:subv-locally-clos}

For $N\ge 0$, we set $\Gm^{N}=(\Qb^{\times})^{N}$ and $\A^{N}=\Qb^{n}$
for the algebraic torus and the affine space over $\Qb$ of dimension
$N$. We will mostly work over the algebraic torus and so, for
simplicity, we will define and study subvarieties and locally closed
subsets in that setting. Nevertheless the notions and properties in
this subsection can be easily transported to the affine space.

A \emph{subvariety} of $\Gm^{N}$ is the zero set of an ideal of 
$\Qb[y_1^{\pm1},\ldots,y_N^{\pm1}]$. Following this convention, a
subvariety is not necessarily irreducible.  More generally, a
\emph{locally closed subset} of $\Gm^{N}$ is the intersection of a
subvariety with a (Zariski) open subset.  The \emph{dimension} of a
locally closed subset is defined as the dimension of its (Zariski)
closure.  A locally closed subset is \emph{irreducible} if its closure
is.

Let $W$ be a locally closed subset of $\Gm^{N}$. An \emph{irreducible
  component} of $W$ is an irreducible locally closed subset $C\subset
W$ that is maximal with respect to inclusion.  An irreducible
locally closed subset $C\subset W$ is an irreducible component of $W$
if and only if $\ov C$ is an irreducible component of~$\ov W$ and
$C=\ov C\cap W$. We denote by $\irr(W)$ the finite collection of the
irreducible components of $W$.  There is an irredundant decomposition
\begin{displaymath}
  W=\bigcup_{C\in \irr(W)}C.
\end{displaymath}

Given a family of Laurent polynomials
$\bfF=\{F_1,\ldots,F_s\}\subset  \Qb[y_1^{\pm1},\ldots,y_{N}^{\pm1}]$, we set
\begin{equation*}
V(\bfF)=  V(F_1,\ldots,F_s)=\{\bfy\in \Gm^{N} \mid
F_{1}({\bfy})=\dots=F_{s}({\bfy})=0\} 
\end{equation*}
for the associated subvariety of $\Gm^{N}$. A family of Laurent
polynomials over $K$ can be considered as a family of Laurent
polynomials over $\Qb$ {\it via} the inclusion $K\hookrightarrow
\Qb$. In particular, such a family of Laurent polynomials defines a
subvariety of $\Gm^{N}$.

We represent a subvariety $W$ of $\Gm^{N}$ by a finite family of
Laurent polynomials $\bfF\subset\Qb[y_1^{\pm1},\ldots,y_{N}^{\pm1}]$
such that $W=V(\bfF)$. More generally, a locally closed subset $W$ of
$\Gm^{N}$ is represented by two finite families $\bfF, \bfG
\subset\Qb[y_1^{\pm1},\ldots,y_{N}^{\pm1}]$ such that
$W=V(\bfF)\setminus V(\bfG)$. A subvariety or a locally closed subset
of $\Gm^{N}$ defined over $K$ are represented similarly by finite
families of Laurent polynomials over $K$.

Let $W$ be a subvariety of $\Gm^{N}$.  If $W$ is a hypersurface, then
there exists $F\in \Qb[\bfy^{\pm1}]$ such that $W= V(F)$ because the
ring $\Qb[\bfy^{\pm1}]$ is a unique factorization domain. If the
codimension of $W$ is higher, then $W$ cannot always be described as a
complete intersection. However, it is possible to describe a
subvariety or, more generally, a locally closed subset, as a finite
union of complete intersections outside hypersurfaces.

\begin{prop-def}
  \label{def:2} Let $W$ be a locally closed subset of $\Gm^N$. Then
  there exists a family of locally closed subsets $W_{j}$, $j=1,\dots,
  t$, given as
$$
W_j=V(F_{j,1},\ldots,F_{j,\codim(W_{j})})\backslash V(G_j)
$$ 
with $F_{j,l}$, $G_j\in \Qb[y_{1}^{\pm 1},\dots, y_{N}^{\pm 1}]$, and satisfying the
following conditions:
\begin{enumerate}
\item \label{item:3} $\codim(W)= \codim(W_{1})<\dots<\codim(W_{t})\le N$;
\item \label{item:4}  $W=\bigcup_{i=1}^{t}W_{j}$.
\end{enumerate}
A family of locally closed subset $(W_{j})_{j}$ as above is called a \emph{complete
  intersection stratification} of $W$.

    If $W$ is defined over $K$, then $(W_{j})_{j}$ can be chosen to be
    defined over $K$ too. 
    In that case, the complete
    intersection stratification is said to be \emph{defined over $K$}.
\end{prop-def}

\begin{proof}
  Set $c=\codim(W)$. We first show that there exist $F_1,\ldots,F_{c},
  G\in \Qb[\bfy^{\pm1}]$ such that $F_1,\ldots,F_{c}$ is a complete
  intersection, $V(G)$ contains no irreducible component of $W$ of
  codimension $c$ and
\begin{equation}
  \label{eq:2}
W\backslash V(G)=V(F_1,\ldots,F_{c})\backslash V(G).  
\end{equation}

  Let $\bfP, \bfQ\subset \Qb [\bfy^{\pm1}]$ such that $W=V(\bfP)\setminus
  V(\bfQ)$ and set
\begin{displaymath}
  F_{l}=\sum_{i}\lambda_{l,i}P_{i}, \quad l=1,\dots, c, \quad \text{and}\quad     Q=\sum_{j}\mu_{j}Q_{j} 
\end{displaymath}
for a choice of $\lambda_{l,i},\mu_{j}\in \Qb$. It can be shown that
if this choice is generic in the sense that the point
$(\bflambda,\bfmu)$ does not lie in a certain hypersurface of the
parameter space, then $F_{1},\dots, F_{c}$ defines a complete
intersection in the complement of $V(\bfQ)$, and the
hypersurface~$V(Q)$ contains no irreducible component of $W$. Assume
that our choice of $\lambda_{l,i},\mu_{j}$ is generic in this sense. 
Then we have that
\begin{displaymath}
 W\setminus V(Q)\subset V(\bfF)\setminus V(Q)  
\end{displaymath}
and, if $C$ is an irreducible component of $W$ of codimension $c$, then
$C\setminus V(Q)$ is an irreducible component of both $ W\setminus
V(Q)$ and $V(\bfF)\setminus V(Q)$. 

Choose $Q'\in \Qb[\bfy^{\pm1}]$ such that the hypersurface $V(Q')$
contains the irreducible components of $ W\setminus V(Q)$ of
codimension $\ge c+1$ and the irreducible components of
$V(\bfF)\setminus V(Q)$ which are not a component of $ W\setminus
V(Q)$, and does not contain any of the other ones. Finally, set
$G=QQ'$. It is not difficult to verify that the Laurent polynomials
$F_{1},\dots, F_{c},G$ satisfy \eqref{eq:2}.

If $\bfP,\bfQ \subset K[\bfy^{\pm1}]$, we can also verify that
$\lambda_{l,i}, \mu_{j}$ and the coefficients of~$Q'$ can be chosen to
lie in $K$, so that $F_{1},\dots, F_{c},G\in K[\bfy^{\pm1}]$.

Now set $W_{1}= V(\bfF)\setminus V(G)$ with $\bfF,G$ as in
\eqref{eq:2}. The intersection $W\cap V(G)$ has codimension $\ge c+1$,
and so we can  construct $W_{2},\dots, W_{t}$ by applying this construction
iteratively. The  family $(W_{j})_{j}$ that we obtain satisfies the
conditions \eqref{item:3} and \eqref{item:4}. It is also clear that,
if $W$ is defined over $K$, then so is $(W_{j})_{j}$.
\end{proof}

Given a locally closed subset $W$ of $\Gm^{N}$, a complete
intersection stratification can be computed either by applying
elimination theory or Gr\"obner basis algorithms. For instance, the computation
of the first piece $W_{1}$ in the case when ${\codim(W)}=2$ has been
worked out in detail in the second author's Ph.D. thesis 
\cite[\S~2.2.2]{Leroux:apl}. The complexity of this procedure will
play no role in our results.


Given a map $\phi\colon \Gm^{n}\to \Gm^{N}$, we denote
by 
\begin{displaymath}
 \phi^{\#}\colon \Qb[y_{1}^{\pm1},\dots,
 y_{N}^{\pm1}]\longrightarrow \Qb[x_{1}^{\pm1},\dots, x_{n}^{\pm1}]   
\end{displaymath}
the associated morphism of algebras. If $\psi\colon \Gm^{N}\to
\Gm^{M}$ is a further map, then 
\begin{equation}\label{eq:23}
  (\psi\circ\phi)^{\#}= \phi^{\#}\circ \psi^{\#}.
\end{equation}
Given an ideal $I\subset \Qb[\bfy^{\pm1}]$, we denote by
$\phi^{\#}(I)$ the ideal of $\Qb[\bfx^{\pm1}]$ generated by the
image of $I$ under $\phi^{\#}$. We have that 
\begin{equation}\label{eq:19}
  V(\phi^{\#}(I))= \phi^{-1}(V(I)).
\end{equation}

\subsection{Homomorphisms and torsion cosets  of algebraic
  tori} \label{sec:tors-cosets-homom} We recall the basic notation and
properties of homomorphisms and algebraic subgroups of tori. We refer
to \cite[\S~3.2]{BombieriGubler:heights} or \cite[Chapter
4]{Zannier:lda} for more details, including the proofs of the
quoted results.

For $n,N\ge 0$, we denote by $\Hom(\Gm^{n},\Gm^{N})$ the set of
homomorphism from $\Gm^{n}$ to $\Gm^{N}$, and by $\Aut(\Gm^{N})$ the
group of automorphisms of $\Gm^{N}$. We denote by $\id_{\Gm^{N}}$ the
identity automorphism of $\Gm^{N}$.

Given a point $\bfx=(x_1,\ldots,x_n)\in\Gm^n$ and a vector
$\bfa=(a_1,\ldots,a_n)\in\Z^n$, we set $\bfx^{\bfa}=x_1^{a_1}\cdots
x_n^{a_n}$. A matrix $A\in \Z^{N\times n}$ with rows 
$\bfa_1,\ldots,\bfa_N\in\Z^n$ gives  a homomorphism
$\varphi_{A}\colon \Gm^{n}\rightarrow \Gm^{N}$ defined, for $\bfx\in
\Gm^{n}$, by
\begin{displaymath}
\varphi_{A}(\bfx)=(\bfx^{\bfa_1},\ldots,\bfx^{\bfa_N}).
\end{displaymath}
The correspondence $A\mapsto \varphi_{A}$ is a bijection between
$\Z^{N\times n}$ and $\Hom(\Gm^{n},\Gm^{N})$.  Given matrices $A\in
\Z^{N\times n}$ and $B\in \Z^{M\times N}$, we have that
\begin{displaymath}
\varphi_{B}\circ  \varphi_{A}= \varphi_{BA}.
\end{displaymath}
We define the \emph{size} of a homomorphism $\varphi\in
\Hom(\Gm^{n},\Gm^{N})$ as 
\begin{displaymath}
  \size(\varphi)=
\max_{j}|\bfa_{j}|,
\end{displaymath}
where $\bfa_{j}$, $j=1,\dots, N$, are the rows of the matrix
corresponding to $\varphi$.

We denote by $1_{\Gm^{N}}=(1,\dots, 1)$ the neutral element of
$\Gm^{N}$.  The subgroup of torsion points of $\Gm^{N}$ agrees with
$\upmu_{\infty}^{N}$.  A \emph{subtorus} of $\Gm^{N}$ is a connected
algebraic subgroup or, equivalently, an algebraic subgroup which is
the image of a homomorphism.  A \emph{torsion coset} of $\Gm^{N}$ is
the translate of a subtorus by a torsion point. A \emph{torsion
  subvariety} of~$\Gm^{N}$ is a finite union of torsion cosets.

A submodule $\Lambda$ of $\Z^{N}$ defines an algebraic subgroup of
$\Gm^{N}$ by
\begin{equation} \label{eq:15}
  H_{\Lambda}=V(\{\bfy^{\bfb}-1 \mid \bfb\in \Lambda\}). 
\end{equation}
The correspondence $\Lambda\to H_{\Lambda}$ is a bijection between the
set of submodules of $\Z^{N}$ and that of algebraic subgroups of
$\Gm^{N}$.  This correspondence reverses dimension, in the sense that
\begin{displaymath}
\codim(H_{\Lambda})= \rank(\Lambda). 
\end{displaymath}
The algebraic subgroup $H_{\Lambda}$ is a subtorus if and only if the
subgroup $\Lambda$ is \emph{saturated}, that is, if and only if
$\Lambda=\R\Lambda\cap \Z^{N}$.

For a submodule $\Lambda$ of $\Z^{N}$, we denote by
$\Lambda^{\sat}=\R\Lambda\cap \Z^{N}$ its \emph{saturation}.  Then
$H_{\Lambda^{\sat}}$ is a subtorus and there is a finite subgroup
$F\subset\upmu_{\infty}^{N}$ of cardinality $[\Lambda^{\sat}:\Lambda]$
such that
\begin{equation} \label{eq:20}
  H_{\Lambda}= F\cdot
  H_{\Lambda^{\sat}}.
\end{equation}

Given a locally closet subset $Y\subset \Gm^{N}$, we denote by $\langle
Y\rangle $ the minimal algebraic subgroup of $\Gm^{N}$ containing
$Y$. Equivalently, 
\begin{displaymath}
  \langle Y\rangle = \bigcap_{H\supset Y}H,
\end{displaymath}
the intersection being over all algebraic subgroups $H$ of $\Gm^{N}$ containing
$Y$. The \emph{multiplicative rank} of $Y$ is defined as
\begin{displaymath}
  \rank(Y)=\dim(\langle Y\rangle).
\end{displaymath}
For instance, a point of $\Gm^N$ has rank $0$ if and only if it is a torsion point,
and it has rank $N$ if and only if its coordinates are multiplicatively
independent. 

The following lemma studies the behavior of the multiplicative rank
under a homomorphism of algebraic tori. 

\begin{lemma} \label{lemm:2} Let $\varphi\colon\Gm^{n}\to \Gm^{N}$ be
  a homomorphism, $Y\subset \im(\varphi)$ a irreducible locally closed
  subset and $C$ an irreducible component of $\varphi^{-1}(Y)$. Then
\begin{displaymath}
  \rank(C)-\dim(C)=   \rank(Y)-\dim(Y).
\end{displaymath}
\end{lemma}

\begin{proof} 
  For each point $\bfxi\in Y$, we have that $\varphi^{-1}(\bfxi)$ is a
  translate of the kernel of $\varphi$ by a point of $\Gm^{n}$.  Hence
  $\dim(\varphi^{-1}(\bfxi))=\dim(\ker(\varphi))$ and, by the theorem
  of dimension of fibers,
\begin{equation}\label{eq:13}
 \dim(\varphi^{-1}(Y))=\dim(Y)+
\dim(\ker(\varphi))
\end{equation}

Set $H=\langle \varphi^{-1}(Y) \rangle$ for the minimal algebraic
subgroup of $\Gm^{n}$ containing $\varphi^{-1}(Y)$. If $\bfxi$ is any
point of $Y$, then $\varphi^{-1}(\bfxi)\subset H$. Since
$\varphi^{-1}(\bfxi)$ is a translate of $\ker(\varphi)$ and $H$ is a
group, it follows that $\ker(\varphi)\subset H$.  Moreover, we have
that $\varphi(H)=\langle Y\rangle$ and so there is an exact sequence
$0\to\ker(\varphi)\to\langle \varphi^{-1}(Y)\rangle\to\langle
Y\rangle\to 0$. We deduce that
  \begin{multline}  \label{eq:12}
\rank(\varphi^{-1}(Y))= \dim(\langle \varphi^{-1}(Y)\rangle) \\ =\dim(\langle Y\rangle) +
  \dim(\ker(\varphi))= \rank(Y)+
  \dim(\ker(\varphi)) .
  \end{multline}

Let $F\subset\upmu_{\infty}^{N}$ be a finite subgroup and
  $T\subset\Gm^{N}$ a subtorus such that $\ker(\varphi)= F\cdot T$ as
  in \eqref{eq:20}
.  Let $C$ be an irreducible component of
  $\varphi^{-1}(Y)$. Then the decomposition of $\varphi^{-1}(Y)$ into
  irreducible components is given by
\begin{displaymath}
  \varphi^{-1}(Y)=\bigcup_{\bfeta\in F}\bfeta C. 
\end{displaymath}
Hence, $\dim(C)= \dim(\varphi^{-1}(Y))$ and $\rank(C)=
\rank(\varphi^{-1}(Y))$. The statement then follows from \eqref{eq:13}
and \eqref{eq:12}.
\end{proof}

We represent torsion cosets as complete intersections of binomials
whose coefficients are roots of unity. In precise terms, let $B$ be a
torsion coset of $\Gm^{N}$ and write $B=\bfeta H_{\Lambda}$ with
$\bfeta\in \upmu_{\infty}$ and $ H_{\Lambda}$ a subtorus of
codimension $r$ corresponding to a saturated submodule $\Lambda\subset
\Z^{N}$. Choose a basis $\bfb_{j}\in \Z^{N}$, $j=1,\dots, r$, of
$\Lambda$.  Then there exist $\xi_{j}\in \upmu_{\infty}$, $j=1,\dots,
r$, such that
\begin{equation} \label{eq:17}
  B=V(\bfy^{\bfb_{1}}-\xi_{1},\dots, \bfy^{\bfb_{r}}-\xi_{r}). 
\end{equation}

The following procedure allows us to compute the preimage of a torsion
coset under a homomorphism. 

 \begin{algorithm}
    \caption{(Preimage of a torsion coset)} \label{alg:11}
    \begin{algorithmic}[1] 
\REQUIRE a homomorphism
$\varphi\colon\Gm^n\rightarrow\Gm^N$ and a torsion coset
$B\subset\Gm^N$.  
\ENSURE either  \og $\varphi^{-1}(B)=\emptyset$\fg{} or a finite
collection  
$\bfE\subset\Qb[\bfx^{\pm1}]$ of binomials with coefficients in $\upmu_{\infty}$.
\STATE \label{alg:18} Let $A\in \Z^{N\times n}$ be the $N\times
n$-matrix of  $\varphi$ and write $B=V(\bfy^{\bfb_{1}}-\xi_{1},\dots,
\bfy^{\bfb_{r}}-\xi_{r})$ with $\bfb_{j}\in \Z^{N}$ and $\xi_{j}\in
\upmu_{\infty}$;
\STATE \label{alg:19} set $\bfc_{j}=\bfb_{j}A\in \Z^{n}$, $j=1,\dots, r$;
\STATE \label{alg:107} compute a
basis $\bfe_{k}\in \Z^{n}$, ${k=1,\dots, t}$, of the submodule generated
by the $\bfc_{j}$'s;
\STATE \label{alg:105}
compute  $\lambda_{k,j},\mu_{j,k}\in \Z$ such that  $\bfe_{k}=\sum_{j}\lambda_{k,j}\bfc_{j}$ and
$\bfc_{j}=\sum_{k}\mu_{j,k}\bfe_{k}$;
\STATE \label{alg:38} set  $\rho_{k}\leftarrow \prod_{j}\xi_{j}^{\lambda_{k,j}}$,
$k=1,\dots, t$;
\IF{$\prod_{k}\rho_{k}^{\mu_{j,k}}=\xi_{j}$ for all $j$} \label{alg:109}
\RETURN \label{alg:106} 
$\bfE\leftarrow\{\bfx^{\bfe_{k}}-\rho_{k}\}_{1\le k\le t}$; 
\ELSE 
\RETURN \og $\varphi^{-1}(B)=\emptyset$\fg{}.
\ENDIF
    \end{algorithmic}
  \end{algorithm}

\begin{lemma}
\label{torsion-coset}
Given a homomorphism $\varphi\colon\Gm^n\rightarrow\Gm^N$ and a
torsion coset $B\subset \Gm^N$, Algorithm~\ref{alg:11} decides if
$\varphi^{-1}(B)\ne\emptyset$. If this is the case, it computes a
finite collection $\bfE=\{\bfx^{\bfe_{k}}-\rho_{k}\}_{1\le k\le t}$ of
binomials in $\Qb[\bfx^{\pm1}]$ with coefficients in $\upmu_{\infty}$
such that $t=\codim( \varphi^{-1}(B))$ and
\begin{equation}\label{eq:3}
  \varphi^{-1}(B)=V(\bfx^{\bfe_{1}}-\rho_{1}, \dots, \bfx^{\bfe_{t}}-\rho_{t}).
\end{equation}

If $n$ and $N$ are bounded, $B$ is given as in \eqref{eq:17} with
$\xi_{j}\in \upmu_{l}$ with $l$ bounded and $\size(\varphi)\le d$,
then each coefficient $\rho_{k}$ has order bounded by~$O(1)$ and each
binomial has degree bounded by $d^{O(1)}$ (by $d$ in the case $n=1$). The
complexity of the algorithm is bounded by
$O(\MM(\log(d))\log(\log(d)))=\wt O(\log(d))$ ops.
\end{lemma}

\begin{proof}
Let notation be as in Algorithm \ref{alg:11}. By
\eqref{eq:19},
\begin{equation}\label{eq:21}
  \varphi^{-1}(B)=V(\bfx^{\bfc_{1}}-\xi_{1},\dots, \bfx^{\bfc_{r}}-\xi_{r}).
\end{equation}
Consider the ideals of $K[\bfx^{\pm1}]$ given by
\begin{displaymath}
  I=(\bfx^{\bfc_{1}}-\xi_{1},\dots,
\bfx^{\bfc_{r}}-\xi_{r}), \quad J= (\bfx^{\bfe_{1}}-\rho_{1}, \dots, \bfx^{\bfe_{t}}-\rho_{t}).
\end{displaymath}
By construction, we have that $\bfx^{\bfe_{k}}-\rho_{k}\in I$ for all
$k$ and, if we set 
$\xi_{j}'=\prod_{k}\rho_{k}^{\mu_{j,k}}$, we have similarly 
that $\bfx^{\bfc_{j}}-\xi_{j}'\in J$ for all $j$. Then 
\begin{equation} \label{eq:22}
  \xi_{j}-\xi_{j}'=
  (\bfx^{\bfc_{j}}-\xi_{j}')-(\bfx^{\bfc_{j}}-\xi_{j}) \in (\bfx^{\bfc_{1}}-\xi_{1},\dots, \bfx^{\bfc_{r}}-\xi_{r}).
\end{equation}
Hence, if there exists $j$ such that $\xi_{j}'\ne \xi_{j}$, it follows
from \eqref{eq:21} and \eqref{eq:22} that
$\varphi^{-1}(B)=\emptyset$. Otherwise, we deduce that $I=J$ and so
$\varphi^{-1}(B)=V(J)$, which proves \eqref{eq:3}. Moreover, the
binomials $\bfx^{\bfe_{k}}-\rho_{k}$, $k=1,\dots, t$, form a complete
intersection because the vectors $\bfe_{k}$, $k=1,\dots, t$, are
linearly independent. Hence $t=\codim(B)$, as stated.

Now suppose that $n$ and $N$ are bounded, $B$ is given as in
\eqref{eq:17} with $\xi_{j}\in \upmu_{l}$ with $l$ bounded, and
$\size(\varphi)\le d$.  The computation of the integers
$\lambda_{k,j},\mu_{j,k}$ in line \ref{alg:105} can be derived from
the Hermite normal form as defined in
\cite[Definition~2.4.2]{Cohen:ccant}, for the matrix with rows
$\bfc_{j}$, $j=1,\dots, r$. This Hermite normal form can be computed
using Algorithm 2.4.5 in~{\it loc.~cit.}~using a bounded number of gcd
computations and multiplications of integers of size bounded by
$\log(d)$. Hence, the integers $\lambda_{k,j},\mu_{j,k}$ have size
bounded by $O(\log(d))$ and the complexity of these steps is bounded
by $O(\MM(\log(d))\log(\log(d)))$.  In particular,
$\deg(\bfx^{\bfe_{k}}-\rho_{k})\le d^{O(1)}$. When $n=1$, we have that
$t=1$ and $\bfe_{1}=\gcd(\bfc_{1},\dots, \bfc_{r})$, and so
$\deg(\bfx^{\bfe_{1}}-\rho_{1})\le d$ in this case.

The computation in line \ref{alg:38} and the verification in line
\ref{alg:109} can be done using repeated squaring. The overall
complexity is bounded by
$O(\MM(\log(d))\log(\log(d)))$, which completes the proof.
\end{proof}

Let $T$ be a torsion coset of $\Gm^{N}$ of codimension 1. By
\eqref{eq:17}, there is a primitive vector~$\bfb\in \Z^{N}$ and
$\xi\in \upmu_{\infty}$ such that
\begin{displaymath}
  T=V(\bfy^{\bfb}-\xi).
\end{displaymath}
The following simple lemma gives a monomial change of coordinates
putting $T$ into a standard position.

\begin{lemma}
  \label{lemm:1} 
  Let $T$ be be a torsion coset of $\Gm^{N}$ of codimension 1 given as
  $ T=V(\bfy^{\bfb}-\xi)$ for a primitive vector $\bfb\in \Z^{N}$ and
  $\xi\in \upmu_{\infty}$.  Let $\bfb_{j}\in \Z^{N}$, $j=1,\dots,N-1$,
  be a family of vectors completing $\bfb$ to a basis of $\Z^{N}$ and
  $\tau\in \Aut(\Gm^{N})$ the automorphism given, for $\bfy\in
  \Gm^{N}$, by
\begin{displaymath}
\tau( \bfy)=
 (\bfy^{\bfb_{1}}, \dots, \bfy^{\bfb_{N-1}},
 \bfy^{\bfb}). 
\end{displaymath}
Then 
  \begin{math}
    \tau(T)= V(y_{N}-\xi).
  \end{math}
\end{lemma}

\begin{proof}
This follows easily from the definitions. 
\end{proof}

\section{Unlikely intersections in algebraic tori}\label{sec:unlik-inters-tori}

In this section, we collect the different results and conjectures on
unlikely intersections in algebraic tori that we will use in the rest
of the paper.

As mentioned in the introduction, the algorithm of Filaseta, Granville
and Schinzel relies on a theorem of Bombieri and Zannier. We
recall the statement of this result, in the refined form later
obtained by these authors together with Masser in \cite[Theorem
4.1]{BombieriMasserZannier:assta}.

\begin{theorem} \label{BZ} Let $W$ be an irreducible subvariety of
  $\Gm^{N}$ of codimension $\ge2$. There exists a constant $c_{W}$
  depending only on $W$ with the following property. Let $\bfzeta\in
  \upmu_{\infty}^{N}$, $\bfa\in \Z^{N}$ and $\alpha\in \Qb^{\times}$. If
  $(\zeta_1\alpha^{a_1},...,\zeta_N\alpha^{a_N}) \in W$, then there exist
  $\bfb\in\Z^{N}\setminus \{\bfzero\}$ such that
\begin{displaymath}
\max_{j}|b_j| \leq c_{W} \quad \text{and} \quad  \prod_{j}(\zeta_j\alpha^{a_j})^{b_j} = 1.
\end{displaymath}
In particular, if $\alpha\notin\upmu_{\infty}$, then $\sum_{j}a_jb_j=0$.
\end{theorem}

In the special case when $\zeta_{j}=1$ for all $j$, the existence of the
bounded non-trivial relation $\sum_{j}a_jb_j =0$ was proposed by
Schinzel in \cite[Conjecture, page 3]{Schinzel:rppt}, anticipating the
Zilber conjecture thirty-seven years before its actual formulation!

\begin{remark} \label{rem:2} The constant $c_{W}$ in Theorem \ref{BZ}
  is effectively calculable, as Zannier has pointed out to us in a
  personal communication. This is also explained in the remark
  following \cite[Theorem 1.6]{BombieriMasserZannier:assta}.
  Unfortunately, the computation of a bound for this constant has not
  been accomplished yet.  Obtaining such a bound would enable our
  algorithm to be made effective in the univariate case (Algorithm
  \ref{alg:27}) and to make explicit the dependence of its complexity
  on the height and number of nonzero terms of the input polynomials.
\end{remark}

The following result is a well-known theorem of Laurent giving a
positive answer to the toric case of the Manin-Mumford
conjecture~\cite{Laurent:ede}. Effective proofs of it can be found in
\cite{Schmidt:hps,Leroux:ctpvdlp}.

\begin{theorem} \label{thm:2} Let $W$ be a subvariety of $\Gm^N$. The
  collection of torsion cosets contained in $W$ which are maximal with
  respect to  inclusion is finite. Equivalently, there exists a
  finite collection $\Theta$ of torsion cosets of $\Gm^{N}$ contained
  in $W$ such that, if $B$ is a torsion coset contained in~$W$, then
  there exists $T\in \Theta$ such that $B\subset T$.
\end{theorem}

In the context of model theory, Zilber proposed in \cite{Zilber:esesc}
the following conjecture on unlikely intersections of subvarieties of
$\Gm^{N}$ with algebraic subgroups. 

\begin{conjecture}[Zilber conjecture] \label{Zilber_Conj1} Let $W$ be
  a subvariety of $\Gm^N$. There exists a finite collection $\Xi$
  of proper algebraic subgroups of $\Gm^{N}$ such that, if $G$ is a
  proper algebraic subgroup of $\Gm^{N}$ and $Y$ an irreducible
  component of $G\cap W$ such that
  \begin{equation} \label{eq:14}
    \dim(Y) > \dim(G) - \codim(W), 
  \end{equation}
then there exists $H\in \Xi$ such that $Y\subset H$.
\end{conjecture}

It is not difficult to see that Theorems \ref{BZ} and  \ref{thm:2} give
this conjecture for the cases when we restrict to algebraic subgroups
$G$ with $\dim(G)=1$ and $\dim(G)=0$, respectively. 

We will use the following reformulation for locally closed subsets of
the Zilber conjecture, which we reinforce by adding the hypothesis that
the collection of torsion cosets can be computed. This is crucial for
our algorithmic applications.

\begin{conjecture} \label{Zilber} Let $W$ be a locally closed subset
  of $\Gm^N$. There exists a finite and  effectively calculable collection
  $\Omega$ of torsion cosets $\Gm^N$ of codimension 1 such that, if
  $B$ is a torsion coset, $C$ an irreducible component of $W$ and
  $Y\subset B\cap C$ an irreducible locally closed subset such that
  \begin{equation}
    \label{eq:6}
    \dim(Y) > \dim(B) - \codim(C), 
  \end{equation}
then there exists $T\in \Omega$ such that $Y\subset T$.
\end{conjecture}

\begin{proposition}
  \label{prop:2}
Conjecture \ref{Zilber_Conj1} is equivalent to the non-effective
version of Conjecture \ref{Zilber}.
\end{proposition}

\begin{proof}
  Suppose that the Zilber conjecture \ref{Zilber_Conj1} holds. Let
  $W\subset \Gm^{N}$ be a locally closed subset and, for each
  irreducible component $C_{j}$ of $W$, let $\Xi_{j}$ the finite
  collection given by this conjecture applied to the subvariety
  $\ov{C_{j}}$. We can assume without loss of generality that each
  proper algebraic subgroup in $\Xi_{j}$ has codimension 1. Consider then
  the finite collection $\Omega$ of torsion cosets of codimension 1
  made of the irreducible components of the algebraic subgroups in the
  collections $\Xi_{j}$.

  Let $B$ be a torsion coset of $\Gm^{N}$ and $\langle B\rangle$ the
  minimal algebraic subgroup containing it. Let
  $C=C_{j_{0}}$ be an irreducible component of $W$ and $Y\subset B\cap
  C$ an irreducible locally closed subset such that \eqref{eq:6}
  holds. Let $Z$ be an irreducible component of $\langle B\rangle\cap \ov C$
  containing $Y$. Then
\begin{displaymath}
  \dim(Z) \ge \dim(Y)> \dim(B) - \codim(C)= \dim(\langle B\rangle) - \codim(\ov C). 
\end{displaymath}
It follows that  there exists $H\in \Xi_{j_{0}}$ such that $Z\subset
H$. Since $Z$ is irreducible, 
there exists $T\in \Omega$ such that $Z\subset T$ and, {\it a
  fortiori}, $Y\subset T$. Hence,  the non-effective version  of
Conjecture \ref{Zilber} holds with this choice of $\Omega$. 

Conversely, suppose that the non-effective version of Conjecture
\ref{Zilber} holds. Let~$W$ be a subvariety of $\Gm^{N}$ and
$\Omega$ the collection of torsion cosets of codimension~1 given by
this conjecture. Then $\Xi=\{ \langle T\rangle \}_{T\in \Omega}$
is a finite collection of  algebraic subgroups of codimension 1, and it is easy to
verify that it satisfies the conditions of the Zilber conjecture~\ref{Zilber_Conj1}.
\end{proof}

\begin{definition} \label{def:3} Let $W$ be a locally closed subset of
  $\Gm^{N}$. An irreducible locally closed subset $Y$ of $W$ is called
  \emph{atypical} if there is a torsion coset $B$ of $\Gm^{N}$ and an
  irreducible component $C$ of $W$ such that $Y\subset B\cap C$ and $
  \dim(Y) > \dim(B) - \codim(C)$.  The \emph{exceptional subset} of
  $W$ is defined as
  \begin{equation*}
W^{\exc}=\bigcup_{Y \text{\rm atypical}} Y.
  \end{equation*}
\end{definition}

\begin{notation}\label{def:4}
  Suppose that Conjecture \ref{Zilber} holds. Given a locally closed
  subset~$W$ of $\Gm^N$, we denote by $\Omega_{W}$ a choice of a
  finite and effectively calculable collection of torsion cosets of
  codimension 1 satisfying the conditions of this conjecture. We also
  write
\begin{displaymath}
  \OmegaW=\bigcup_{T\in \Omega_{W}} T.
\end{displaymath}
\end{notation}

Conjecture \ref{Zilber} implies that the exceptional set of a locally
closed subset $W$ is contained in the torsion subvariety $\OmegaW$. In
particular, if $W$ is not contained in a proper torsion subvariety of
$\Gm^{N}$, then $W^{\exc}$ is a proper subset of
$W$.

\section{The univariate case} \label{dim1}

In this section, we present an algorithm that, given a system of
univariate polynomials with coefficients in the number field $K$, of
degree bounded by $d$ and bounded height and number of nonzero
coefficients, computes a single polynomial defining the same zero set
with $\wt O(\log(d))$ ops. Our approach is inspired by the one
in~\cite{FilasetaGranvilleSchinzel:igcdasp}, but it is simpler and
more geometric.

The following subroutine is one of the main components of the
algorithm. It tests whether a subtorus of codimension 1 contains the
image of a homomorphism and, if this is the case, reduces the
dimension of the problem by intersecting the variety
under consideration with that subtorus.

 \begin{algorithm} 
    \caption{(Reduction of dimension for ideals)} \label{alg:20}
    \begin{algorithmic}[1]
\REQUIRE a homomorphism $\varphi\in
\Hom(\Gm^{n},\Gm^{N})$, a subtorus $T\subset\Gm^{N}$ of codimension 1,
and a family of Laurent polynomials $F_{i}\in K[y_{1}^{\pm1},\dots,
y_{N}^{\pm1}]$, $i=1,\dots, s$.
\ENSURE   either \og $\im(\varphi)\not\subset T$ \fg{}, or a homomorphism $\wt\varphi\in
\Hom(\Gm^{n},\Gm^{N-1})$ and a family of Laurent polynomials $\wt F_{i}\in K[y_{1}^{\pm1},\dots,
y_{N-1}^{\pm1}]$, $i=1,\dots, s$.
\STATE \label{alg:25} Let $A\in \Z^{N\times n}$ be the $N\times
n$-matrix associated to $\varphi$ and $\bfb\in
\Z^{N}$ a primitive vector such that
$T=V(\bfy^{\bfb}-1)$;
\IF{$\bfb A=\bfzero$} \label{alg:26}
\STATE \label{alg:22} choose $\tau\in \Aut(\Gm^{N})$ such that $\tau(T)=
V(y_{N}-1)$ as in Lemma \ref{lemm:1};
\STATE \label{alg:23} let $\iota\colon \Gm^{N-1}\to\Gm^{N}$ be the standard
inclusion identifying $\Gm^{N-1}$ with 
the hyperplane $V(y_{N}-1)$, and $\pi\colon \Gm^{N}
\to\Gm^{N-1}$ the projection onto the first $N-1$ coordinates;
\RETURN \label{alg:24} 
$\wt\varphi\leftarrow \pi\circ \tau \circ \varphi$ and $\wt F_{i}\leftarrow
(\tau^{-1}\circ\iota)^{\#}(F_{i})$, $i=1,\dots, s$;
\ELSE
\RETURN  \og $\im(\varphi)\not\subset T$ \fg{};
\ENDIF
    \end{algorithmic}
  \end{algorithm}

\begin{lemma}
\label{E2c0-i}
For a given input $(\varphi, T, \bfF)$, Algorithm \ref{alg:20} stops
after a finite number of steps. It decides if $\im(\varphi)\subset T$
and, if this is the case, its output $(\wt \varphi, \bfF)$ satisfies
\begin{equation} \label{eq:16}
  \wt \varphi^{\#}(\wt \bfF) =   \varphi^{\#}(\bfF).
\end{equation}
In particular, $\wt \varphi^{-1}(V(\wt \bfF))= \varphi^{-1}(V(\bfF))$. 

If $\size(\varphi)\le d$, then $\size(\wt\varphi)\le O_T(d)$ and each
$\wt F_{i}$ has degree, height and number of nonzero coefficients
bounded by $O_{T,\bfF}(1)$. The complexity of the algorithm is bounded
by $O_{T}(\log (d))+ O_{T,\bfF}(1)$~ops.
\end{lemma}

\begin{proof} Let notation be as in Algorithm \ref{alg:20}. Under the
  correspondence in \eqref{eq:15}, the subtorus $\im(\varphi)$ is
  associated to the kernel of the linear map $A^{\trans}\colon
  \Z^{N}\to \Z^{n}$, for the matrix $A$ in line
  \ref{alg:25} of the algorithm. Then $T=V(\bfy^{\bfb}-1)$ contains
  the image of $\varphi$ if and only if $\bfb\in \ker(A^{\trans})$ or,
  equivalently, if $\bfb A =\bfzero$. Hence, the algorithm decides
  correctly whether this holds. 

To prove \eqref{eq:16}, assume that $\im(\varphi)\subset T$. Consider
the diagram
\begin{displaymath}
  \xymatrix{\Gm^{N} \ar@/^1pc/[r]^{\tau} 
&  \Gm^{N} \ar@/^1pc/[l]_{\tau^{-1}}
\ar@/^1pc/[r]^{\pi}
& \Gm^{N-1} \ar@{=}[d]\ar@/^1pc/[l]_{\iota}\\
    T  \ar@{^{(}->}[u]\ar@/^1pc/[r] &  V(y_{N}-1) \ar@/^1pc/[r]
    \ar@{^{(}->}[u] \ar@/^1pc/[l]& \Gm^{N-1} \ar@/^1pc/[l]}
\end{displaymath}
where the maps in the second line are induced by the ones in the first
line. These maps in the second line are isomorphisms.  Considering
the corresponding maps of $K$-algebras, we deduce that
\begin{displaymath}
  (\bfF)+(\bfy^{\bfb}-1) = (\pi\circ\tau)^{\#}
  ((\tau^{-1}\circ\iota)^{\#}((\bfF)+(\bfy^{\bfb}-1))
=(\pi\circ\tau)^{\#}(\wt \bfF)
\end{displaymath}
because $(\tau^{-1}\circ\iota)^{\#}(\bfy^{\bfb}-1)=
\iota^{\#}(y_{N}-1)=0$. We also have that
$\varphi^{\#}(\bfy^{\bfb}-1)=0$ because $\im(\varphi)\subset T$.
Using the functoriality \eqref{eq:23}, it follows that
\begin{displaymath}
  \varphi^{\#}(\bfF)=
  \varphi^{\#}((\bfF)+(\bfy^{\bfb}-1))=
  \varphi^{\#}((\pi\circ\tau)^{\#}(\wt \bfF))= \wt \varphi^{\#}(\wt \bfF),
\end{displaymath}
as stated. Clearly, this implies that $\wt \varphi^{-1}(V(\wt \bfF))=
\varphi^{-1}(V(\bfF))$.

Now suppose that $\size(\varphi)\le d$.  The automorphism $\tau$ in
line \ref{alg:22} depends only on $T$. Hence, the construction of $\wt
\varphi$ in line \ref{alg:24} implies that $\size(\wt \varphi)\le
O_{T}(d)$.  The construction of each $\wt F_{i}$, also in line
\ref{alg:24}, consists of composing $F_{i}$ with a homomorphism
depending only on $T$. Hence, the number of nonzero coefficients of
$\wt F_{i}$ is bounded by that of $F_{i}$, and its degree and height
are bounded by $O_{T,\bfF}(1)$.

The verification $\bfb A=\bfzero$ in line \ref{alg:26} costs
$O_{T}(\log(d))$ ops using classical multiplication of integers. The
construction of $\tau$ in line~\ref{alg:22} uses $O_{T}(1)$ ops. The
computations in line~\ref{alg:24} of $\wt \varphi$ and $\wt F_{i}$,
$i=1,\dots, s$, take $O_{T}(\log(d))$ ops and $O_{T,\bfF}(1)$ ops,
respectively, using classical multiplication. Hence, the overall
complexity of the algorithm is bounded by $O_{T}(\log (d))+
O_{T,\bfF}(1)$~ops.
\end{proof}

\begin{remark} \label{rem:8}
The Laurent polynomials $\wt F_{i}$, $i=1,\dots, s$, in Algorithm
\ref{alg:20}, line \ref{alg:24}, can be written down, in more explicit terms,
as
\begin{displaymath}
  \wt F_{i}= F_{i}\circ \tau^{-1}(y_{1},\dots, y_{N-1},1) .
\end{displaymath}
\end{remark}

The following procedure computes the non-torsion points in the zero
set of a system of univariate polynomials, and gives the first half of
the algorithm. It is based on the strategy described in (\ref{eq:25}). In
the univariate setting, we apply Theorem~\ref{BZ}
to successively intersect the linear subvariety $W$ until all
non-torsion points lie in a hypersurface (while loop between
lines \ref{alg:39} and \ref{alg:40} in Algorithm \ref{alg:27}
below). Once this is achieved, the hypersurface is described by a
single polynomial that we obtain through a gcd computation, and the
non-torsion points are obtained as the inverse image under a
homomorphism of this hypersurface (lines~\ref{alg:41} and
\ref{alg:42}).

For a Laurent polynomial $p\in K[x^{\pm1}]\setminus \{0\}$, we denote by $\ord(p)$
the maximal integer $m$ such that $x^{-m}p\in K[x]$. 

 \begin{algorithm}     \caption{(Non-torsion points)} \label{alg:27}
    \begin{algorithmic}[1]
\REQUIRE  a family of polynomials $f_{i}\in K[x]$,
$i=1,\dots, s$.
\ENSURE  a polynomial $p_{1}\in K[x]$.
\STATE  Write  $f_i=\sum_{j=1}^N
   \alpha_{i,j}x^{a_j}$ with $\alpha_{i,j}\in K$ and
   $a_{j}\in  \bigcup_{i}\supp(f_{i})$;
\STATE \label{alg:43} set $F_{i}\leftarrow \sum_{j=1}^N \alpha_{i,j}y_j$, $i=1,\dots,
s$ and $\bfF\leftarrow (F_{1},\dots, F_{s})$;
\STATE set $k\leftarrow 0$   and let $\varphi\in
\Hom(\Gm, \Gm^{N})$ be the homomorphism corresponding to the exponents
$a_{j}\in\N$, $j=1,\dots, N$;
\WHILE{$k< N$} \label{alg:39}
\STATE  \label{alg:44} for each irreducible component $W$  of $V(\bfF)$ of codimension $\ge2$, compute
$c_{W}$ as in  Theorem~\ref{BZ} and set
$\Phi\leftarrow \{V(\bfy^{b}-1)\mid \bfb\in \Z^{N} \text{ primitive
  such that } \max_{j} |b_{j}|\le \max_{W}c_{W}\}$;
\WHILE{$\Phi\ne \emptyset$} 
\STATE \label{alg:45} choose $T\in \Phi$ and apply Algorithm \ref{alg:20} to $(\varphi, T, \bfF)$; 
\IF{$\im(\varphi)\subset T$}
\STATE \label{alg:32} set $\bfF\leftarrow \wt \bfF$,
$\varphi\leftarrow \wt \varphi$, $k\leftarrow k+1$, $\Phi\leftarrow
\emptyset$;
\ELSE 
\STATE set $\Phi\leftarrow \Phi\setminus \{T\}$; 
\ENDIF
\ENDWHILE
\ENDWHILE \label{alg:40}
\STATE \label{alg:41} set $p\leftarrow \varphi^{\#}(\gcd(F_{1},\dots, F_{s}))$; 
\RETURN \label{alg:42} $p_{1}\leftarrow x^{-\ord(p)}p$.
    \end{algorithmic}
  \end{algorithm}

  \begin{theorem} \label{thm:4}
Given $f_{i}\in K[x]$,
$i=1,\dots, s$, Algorithm \ref{alg:27} stops after a finite number of
steps and computes  $p_{1}\in K[x]$ such that 
 \begin{displaymath}
    p_{1}|\gcd(f_{1},\dots, f_{s})\quad  \text{ and }\quad  
V\Big(\frac{\gcd(f_{1},\dots, f_{s})}{p_{1}}\Big)\subset \upmu_{\infty}.
  \end{displaymath}

  If $s$ is bounded and each $f_{i}$ has degree bounded by $d$ and
  bounded height and number of nonzero coefficients, then $p_{1}$ has
  degree bounded by $d$, and height and number of nonzero coefficients
  bounded by $O(1)$. The complexity of the algorithm is bounded by
  $O(\log (d))$~ops.
  \end{theorem}

  \begin{proof} For $i=0,\dots, s$, the initial value of  $F_{i}$ at line \ref{alg:43} satisfies
    \begin{equation}\label{eq:26}
      \varphi^{\#}(F_{i})=f_{i}.
    \end{equation}
    By Lemma~\ref{E2c0-i}, this holds also for the updated values of
    $F_{i}$ and $\varphi$ in line \ref{alg:32}. Hence, the equality
    \eqref{eq:26} holds for the final value of $F_{i}$ as in line
    \ref{alg:41}.  For the rest of this proof, we denote by  $F_{i}$
    this final value. 

    Set $P=\gcd(F_{1},\dots, F_{s})$. We have that
\begin{math}
 (F_{1},\dots,
    F_{s})\subset (P)\subset
    K[y_{1}^{\pm1},\dots, y_{N-k}^{\pm1}],
\end{math}
The equality
\eqref{eq:26} implies that $(f_{1},\dots, f_{s})\subset (p) \subset
K[x^{\pm1}]$ and so 
\begin{displaymath}
  p_{1}|\gcd(f_{1},\dots,
f_{s}) \quad \text{ in } K[x].
\end{displaymath}

Set $\bff=\{f_{1},\dots, f_{s}\}$ and $\bfF=\{F_{1},\dots, F_{s}\}$
for short. Let $\alpha\in V(\bff)\setminus \upmu_{\infty}$, so that
$\varphi(\alpha)\in V(\bfF) $. 
Let $W$ an irreducible component of $V(\bfF)$ such that
\begin{displaymath}
\varphi(\alpha)=(\alpha^{a_{1}},\dots, \alpha^{a_{N}})\in W.
\end{displaymath}
By Theorem \ref{BZ}, $W$ is necessarily of codimension~1. Otherwise,
there would exist $T\in \Phi$ such that $\im(\varphi)\subset T$ but,
by construction, this is not possible.

This discussion implies that the ideal $(\bfF)\subset
K[y_{1}^{\pm1},\dots, y_{N-k}^{\pm1}]$ becomes principal when
restricted to a suitable neighborhood $U\subset\Gm^{N-k}$ of
$\varphi(V(\bff)\setminus \upmu_{\infty})$. Hence, $(\bfF)=(P)$ on
that neighborhood. We deduce that $\varphi^{-1}(U)$ is a neighborhood
of $V(\bff)\setminus \upmu_{\infty}$ and $(\bff)= \varphi^{\#}(p)$ on
$\varphi^{-1}(U)$.  Thus, $ V({\gcd(f_{1},\dots, f_{s})}/p_{1}) \cap
\varphi^{-1}(U)=\emptyset$ or, equivalently,
\begin{displaymath}
  V\Big(\frac{\gcd(f_{1},\dots,
  f_{s})}{p_{1}}\Big) \subset \upmu_{\infty}. 
\end{displaymath}
This completes the proof of the first part of the statement.

For the second, part, assume that $s$ is bounded and that the
$f_{i}$'s have degree bounded by $d$, and bounded height and number of
nonzero coefficients. The construction of $P=\gcd(F_{1},\dots, F_{s})$
does not depend on the exponents $a_{j}$. Hence, the height and number
of nonzero coefficients of $P$ is bounded, and so this also holds for
$p_{1}$. The fact that $p_{1}$ divides $\gcd(f_{1},\dots, f_{s})$
implies that $\deg(p_{1})\le d$.  The list of linear forms $\bfF$ in
line \ref{alg:43} does not depend on the exponents $a_{j}$ and so,
\emph{a fortiori}, it is independent on the degree of the input
polynomials $f_{i}$. Therefore, the computations in lines \ref{alg:44}
and \ref{alg:41} cost $O(1)$ ops. By Lemma \ref{E2c0-i}, the
computation in line \ref{alg:45} costs $O(\log(d))$ ops. Since the
number of iterations in the while loop between lines \ref{alg:39} and
\ref{alg:40} is bounded, we conclude that the overall complexity of
the algorithm is bounded by $O(\log(d))$, as stated.
  \end{proof}

  The following procedure computes the torsion points in the zero set
  of a system of univariate polynomials, and completes the
  algorithm. It consists in considering the finite collection of
  maximal torsion cosets given by Theorem \ref{thm:2} for the linear
  subvariety~$W$ as in \eqref{eq:25}, and compute with Algorithm
  \ref{alg:11} its inverse image with respect to the homomorphism
  $\varphi$.

 \begin{algorithm}     \caption{(Torsion points)} \label{alg:28}
    \begin{algorithmic}[1]
\REQUIRE  a family of polynomials $f_{j}\in K[x]$,
$j=1,\dots, s$.
\ENSURE  a polynomial $p_{2}\in K[x]$.
\STATE  Write  $f_i=\sum_{j=1}^N
   \alpha_{i,j}x^{a_j}$ with $\alpha_{i,j}\in K$ and
   $a_{j}\in  \bigcup_{i}\supp(f_{i})$;
\STATE set $F_{i}\leftarrow \sum_{j=1}^N \alpha_{i,j}y_j$, $i=1,\dots,
s$;
\STATE \label{alg:48} set $W\leftarrow V(F_{1},\dots, F_{s})\subset
\Gm^{N}$, $\varphi\in
\Hom(\Gm, \Gm^{N})$ the homomorphism corresponding to the exponents
$a_{j}\in \N$, $j=1,\dots,N$, and $\Lambda\leftarrow \emptyset$;
\STATE \label{alg:46} compute the collection $\Theta$ of
maximal torsion cosets of $W$ from Theorem \ref{thm:2};
\FOR{$B\in \Theta$} \label{alg:49}
\STATE apply Algorithm \ref{alg:11} to the pair $(\varphi,B)$;
\IF{$\varphi^{-1}(B)\ne \emptyset$} 
\STATE \label{alg:47} let $\bfE$ be the output of Algorithm \ref{alg:11},
write $\bfE=\{x^{b}-\xi\}$ with $b\in \N$ and
$\xi\in \upmu_{\infty}$, and add the  binomial $x^{b}-\xi$ to $\Lambda$; 
\ENDIF
\ENDFOR \label{alg:50}
\RETURN $p_{2}\leftarrow \prod_{g\in \Lambda} g$.
    \end{algorithmic}
  \end{algorithm}

  \begin{theorem} \label{thm:5} Given $f_{i}\in K[x]$, $i=1,\dots, s$,
    Algorithm \ref{alg:28} stops after a finite number of steps and
    computes $p_{2}\in K[x]$ such that
 \begin{displaymath}
V(p_{2})= 
V(f_{1},\dots, f_{s}) \cap \upmu_{\infty}.
  \end{displaymath}

  If $s$ is bounded and each $f_{i}$ has degree bounded by $d$ and
  bounded height and number of nonzero coefficients, then $p_{2}$ has
  degree bounded by $O(d)$ and height and number of nonzero
  coefficients bounded by $O(1)$. The complexity of the algorithm is
  bounded by $O(\MM(\log(d))\log(\log(d)))=\wt O(\log(d))$ ops.
  \end{theorem}

  \begin{proof}
Write $\bff=(f_{1},\dots, f_{s})$ for short and let $\Theta$ be as in line \ref{alg:46} of the algorithm. Then
\begin{displaymath}
  V(\bff)\cap \upmu_{\infty}= \bigcup_{B\in
    \Theta}\varphi^{-1}(B).
\end{displaymath}
With notation as in line \ref{alg:47}, we have that $\varphi^{-1}(B)=
V(x^{b}-\xi)$ because of \eqref{eq:3}. It follows that 
$$ 
V(\bff)\cap
\upmu_{\infty}=\bigcup_{g\in \Lambda}V(g)=V(p_{2}).
$$
The collection $\Theta$ is invariant under $K$-automorphisms of $\Qb$ and
so is $p_{2}$. It follows that $p_{2}\in K[x]$, completing the proof
of the first part of the statement.

The construction of $W$ in line \ref{alg:48} and, {\it a fortiori},
that of $\Theta$, do not depend on the exponents $a_{j}$. Hence, the
size of $\Theta$ is bounded by $O(1)$ and its computation costs $O(1)$
ops. Hence, the number of steps in the for loop between lines
\ref{alg:49} and \ref{alg:50} is bounded by $O(1)$. The second part of
the statement then follows easily from Lemma \ref{torsion-coset}.
  \end{proof}

Putting together Theorems \ref{thm:4} and \ref{thm:5}, we obtain the
following more general and precise version of Theorem \ref{thm:1} in
the introduction.

\begin{theorem}
\label{gcd}
Given $f_{i}\in K[x]$, $i=1,\dots, s$,
Algorithms \ref{alg:27} and \ref{alg:28} compute $p_{1},p_{2}\in K[x]$
such that $p_{1}|\gcd(f_{1},\dots, f_{s})$,
 \begin{displaymath}
V(p_{1})\setminus
    \upmu_{\infty}= V(\gcd(f_{1},\dots, f_{s})) \setminus
    \upmu_{\infty} \quad  \text{ and }\quad  
V(p_{2})=V({\gcd(f_{1},\dots, f_{s})}) \cap \upmu_{\infty}.
  \end{displaymath}

  If $s$ is bounded and each $f_{i}$ has degree bounded by $d$,
  bounded height and number of nonzero coefficients, then
  $\deg(p_{1})\le d$ and $\deg(p_{2}) \leq O(d)$, and the height and
  number of nonzero coefficients of both $p_{1}$ and $p_{2}$ are
  bounded by $O(1)$. The complexity of the procedure is bounded by
  $O(\MM(\log(d))\log(\log(d)))=\wt O(\log(d))$ ops.
\end{theorem}

\section{The multivariate case}
\label{general}

In this section, we present the procedure for the reduction of
overdetermined systems of multivariate polynomial equations.  We first
give a simple algorithm which allows us to treat the components of top
multiplicative rank.  Strictly speaking, this procedure is not needed
to treat the general case, but it is considerably simpler and serves as a
first approach. After this, we give the main procedure (Algorithm
\ref{alg:5}) and prove Theorem~\ref{main-s} in the introduction.

We will express our algorithms in terms of geometrical objects to
avoid the burden of their syntactical treatment in terms of families
of Laurent polynomials. As before, we denote by $K$ a number field
together with an inclusion into $\Qb$.

\subsection{The weakly transverse case}\label{sec:weakly-transv-case}

The following is a reformulation of Algorithm~\ref{alg:20} in
terms of locally closed subsets.
 \begin{algorithm} 
    \caption{(Reduction of dimension for locally closed subsets)} \label{alg:33}
    \begin{algorithmic}[1]
\REQUIRE 
$(n,N,\varphi,T,W)$ where  $n$ and $N$ are positive integers, $\varphi\in
\Hom(\Gm^{n},\Gm^{N})$ is a homomorphism, $T\subset\Gm^{N}$ is a subtorus of codimension 1,
and  $W\subset \Gm^{N}$ is a locally closed subset.
\ENSURE   either \og $\im(\varphi)\not\subset T$ \fg{}, or a homomorphism $\wt\varphi\in
\Hom(\Gm^{n},\Gm^{N-1})$ and a locally closed subset $\wt W\subset \Gm^{N-1}$.
\STATE Take $\bfF,\bfG\subset \Qb[y_{1}^{\pm1},\dots, y_{N}^{\pm1}]$
such that $W=V(\bfF)\setminus V(\bfG)$;
\STATE apply Algorithm \ref{alg:20} to $(\varphi,T, \bfF)$ and to
$(\varphi,T, \bfG)$ and, 
if $\im(\varphi)\subset T$, set  $(\wt \varphi,\wt \bfF)$ and
$(\wt \varphi,\wt \bfG)$ for its output;
\RETURN $\wt W\leftarrow V(\wt \bfF)\setminus V(\wt \bfG)$.
    \end{algorithmic}
  \end{algorithm}

  \begin{lemma} \label{lemm:3} For a given input $(\varphi, T, W)$,
    Algorithm \ref{alg:33} stops after a finite number of steps. It
    decides whether $\im(\varphi)\subset T$ and, if this is the case, its
    output satisfies
\begin{equation} \label{eq:11}
  \wt \varphi^{-1}(\wt W) =   \varphi^{-1}(W).
\end{equation}

If $\size(\varphi)\le d$, then $\size(\wt\varphi)\le O_T(d)$ and the
Laurent polynomials defining $\wt W$ have degree, height and number of
nonzero coefficients bounded by $O_{T,\bfF}(1)$. The complexity of the
algorithm is bounded by $O_{T,W}(\log (d))$~ops.
  \end{lemma}

  \begin{proof}
This follows readily from Lemma \ref{E2c0-i}.
  \end{proof}

  \begin{definition} \label{def:1} Let $X$ be an irreducible locally
    closed subset of $\Gm^{n}$. Following
    Viada~\cite{Viada:icutctspec}, we say that $X$ is {\it weakly
      transverse} if it is not contained in any proper torsion coset
    or, equivalently, if $\rank(X)=n$.
\end{definition}

\begin{remark} \label{rem:4} An irreducible locally closed subset
  $X\subset \Gm^{n}$ is weakly transverse if and only if it is not
  atypical as a subset of itself. If $\Omega_{X}$ denotes a finite
  collection of torsion cosets of $\Gm^{n}$ of codimension 1 as in
  Notation \ref{def:4}, this is equivalent to the condition that
  $X\not\subset \bigcup\Omega_{X}$.
\end{remark}

The following procedure is the natural generalization of Algorithm
\ref{alg:27} to the multivariate case.

 \begin{algorithm}     \caption{(Reduction of weakly transverse components)} \label{alg:200}
    \begin{algorithmic}[1]
\REQUIRE  a subvariety $V\subset \Gm^{n}$  defined by
a linearly independent family of Laurent polynomials in
$\ov \Q[x_{1}^{\pm1},\dots, x_{n}^{\pm1}]$;
\ENSURE  a finite collection $\Gamma$ of locally closed subsets
 of $\Gm^{n}$.
\STATE Let $f_i$, $i=1,\dots, s$, be the Laurent polynomials
defining~$V$ and write $f_i=\sum_{j=1}^N
\alpha_{i,j}\bfx^{\bfa_j}$ with $\alpha_{i,j}\in \Qb$ and
$\bfa_{j}\in \bigcup_{i}\supp(f_{i})$;
\STATE \label{alg:203} set $k\leftarrow 0$, $W\leftarrow
V(\sum_{j=1}^N \alpha_{1,j}y_j, \dots, \sum_{j=1}^N
\alpha_{s,j}y_j)\subset\Gm^N$,  $t\leftarrow 1$,  $W_{1}\leftarrow
W$, and $\varphi\in \Hom(\Gm^{n}, \Gm^{N})$ the
homomorphism corresponding to the exponents $\bfa_{j}\in \Z^{n}$, $j=1,\dots, N$;
\STATE \label{alg:35} compute the collection  $\Omega_{W_{1}}$ of  torsion cosets of
codimension 1  (Notation  \ref{def:4}) and set  
$\Phi\leftarrow \{T\in\Omega_{W_{1}}\mid T
\text{ is a subtorus}\}$;
\WHILE{$\exists\  T\in \Phi$ such that $\im(\varphi)\subset T$} \label{alg:36}
\STATE \label{alg:14} 
apply Algorithm \ref{alg:33} with input $(n, N-k, \varphi, T,
{W})$ and denote by $(\wt \varphi,\wt W)$  its output; 
\STATE \label{alg:37} set
$W\leftarrow \wt W$, $\varphi\leftarrow
\wt \varphi$, $k\leftarrow k+1$;
\STATE \label{alg:15} compute a complete
intersection stratification $(W_{j})_{1\le j\le t}$ of $W$
(Proposition-Definition~\ref{def:2});
\STATE \label{alg:16} compute the collections $\Omega_{W_{j}}$, $j=1,\dots,t$; 
\STATE \label{alg:51} set 
$\Phi\leftarrow  \bigcup_{j=1}^{t}\{T\in\Omega_{W_{j}}\mid T
\text{ is a subtorus}\}$; 
\ENDWHILE \label{alg:202}
\RETURN \label{alg:17}  $\Gamma\leftarrow
\{\varphi^{-1}(W_{j})\}_{1\le j\le t}$.
    \end{algorithmic}
  \end{algorithm}


\begin{theorem}
\label{main-wt}
Assume that Conjecture \ref{Zilber} holds. Given a subvariety
$V\subset\Gm^n$, Algorithm \ref{alg:200} 
stops after a finite number of steps and its output satisfies:
\begin{enumerate}
\item \label{item:12}
\begin{equation*}
V=\bigcup_{Z\in \Gamma} Z;
\end{equation*}

\item \label{item:13} each $Z\in \Gamma$ is given as the zero set of
  $l_{Z}$ Laurent polynomials in the complement of the zero set of a
  further Laurent polynomial. Moreover, if $C$ is an irreducible
  component of $Z$ that is weakly transverse, then $\codim(C)=l_{Z}$.
\end{enumerate}

If $n$ is bounded and $V$ is defined over $K$ by a bounded number of
Laurent polynomials of degree $\leq d$, of bounded height and number
of nonzero coefficients, then the cardinality of $\Gamma$ is bounded
by $O(1)$, the Laurent polynomials defining each $Z\in \Gamma$ have
coefficients in $K$, degree bounded by ${O(d)}$, and height and number
of nonzero coefficients bounded by $O(1)$. The complexity of the
algorithm is bounded by $O(\log (d))$~ops.
\end{theorem}

\begin{proof}
  If Conjecture \ref{Zilber} holds, then the computation of the
  collections $\Omega_{W_{j}}$ in lines~\ref{alg:35} and \ref{alg:16}
  can be done and the algorithm makes sense.  At each while loop
  (lines~\ref{alg:36} to~\ref{alg:202}), the value of the variable $k$
  in line \ref{alg:37} increases by one. Hence, this while loop cannot
  be repeated more than $N$ times since when $k=N$, the collection
  $\Phi$ in line \ref{alg:51} is empty. Hence, the algorithm stops
  after a finite number of steps.

  To prove \eqref{item:12}, we first show that, after each while loop, the subvariety
  $W\subset\Gm^{N-k}$ and the homomorphism $\varphi\in
  \Hom(\Gm^{n},\Gm^{N-k})$ satisfy
\begin{equation}
  \label{eq:9}
V=\varphi^{-1}(W).
\end{equation}
By construction, this is clear for the initial values of $W$ and
$\varphi$ in line \ref{alg:203}.  Now suppose that \eqref{eq:9} holds
at the start of the while loop. If this loop is actually executed,
then there exists a subtorus $T\in \Phi$ such that
$\im(\varphi)\subset T$. With notation as in line \ref{alg:14}, by
\eqref{eq:11} it follows that $\varphi^{-1}(W) = \wt \varphi^{-1}(\wt
W) $. We conclude that \eqref{eq:9} also holds for the updated values
of $W$ and $\varphi$ in line \ref{alg:37}.

For the rest of this proof, we denote by $W\subset \Gm^{N-k}$ the
final value of this variable after the last execution of the while
loop, and $(W_{j})_{j}$ its corresponding complete intersection
stratification. Then $W=\bigcup_{j}W_{j}$ and so
\begin{displaymath}
V=\bigcup_{j}\varphi^{-1}(W_{j}) = 
\bigcup_{Z\in \Gamma} Z.
\end{displaymath}

The first part of \eqref{item:13} follows from the definition of a
complete intersection stratification and the construction of $\Gamma$
in line \ref{alg:17}.  We now prove the second part. Let $C$ be an
irreducible component of $Z\in \Gamma$ that is weakly transverse and
suppose that $\codim(C)<l_{Z}$.  Let  $W_{j}$ be a stratum in the
complete intersection stratification of $W$ of
codimension $l_{Z}$ and such that $Z=\varphi^{-1}(W_{j})$. Let $Y$ be an
irreducible component of $W_{j}\cap\im(\varphi)$ such that $C$ is an
irreducible component of $\varphi^{-1}(Y)$.  Applying Lemma
\ref{lemm:2} and the fact that $C$ is weakly transverse, we obtain
that
\begin{equation} \label{eq:27}
  \codim(C)= \rank(C)-\dim(C)=\rank(Y)-\dim(Y).
\end{equation}
The fact that $C$ is weakly transverse also implies that $\langle Y \rangle
= \im(\varphi)$ and so $\rank(Y)= \dim(\im(\varphi))$.  Let $\wt C$ be
an irreducible component of $W_{j}$ containing $Y$, so that
$Y\subset\im(\varphi)\cap \wt C$. From \eqref{eq:27}, we deduce that
\begin{displaymath}
\dim(Y)>\dim(\im(\varphi)) - \codim(\wt C).
\end{displaymath}
Conjecture \ref{Zilber} then implies that there exists $T\in
\Omega_{W_{j}}$ such that $Y\subset T$. It follows that
$\im(\varphi)\subset T$ and that $T$ is a subtorus. But this cannot
happen since, otherwise, the while loop would not been
terminated. We deduce that $\codim(C)\ge l_{Z}$.

On the other hand, since $Z$ is defined by $l_{Z}$ equations in the
complement of a hypersurface, it follows that $\codim(C)= l_{Z}$,
proving the second part of \eqref{item:13}.

The subvariety $W$ in line~\ref{alg:203} does not depend on the
exponents $a_j$ and so, \emph{a fortiori}, is independent on the bound
$d$ for the degrees of the polynomials defining $V$. The bounds for
the size of the output and the complexity of the algorithm then follow
easily from Lemma ~\ref{E2c0-i}.
\end{proof}

With notation and assumptions as in Theorem \ref{main-wt},  denote
by $\Gamma_{\max}$ the subset of~$\Gamma$  consisting of the
locally closed subsets $Z\in \Gamma$ such that $\ov Z$ is maximal with
respect to inclusion.  Clearly,
\begin{equation*}
V=\bigcup_{Z\in \Gamma_{\max}} \ov Z.
\end{equation*}
Hence, given an irreducible component $C$ of $V$ that is weakly
transverse, there exists $Z\in \Gamma_{\max}$ such that $C$ is an
irreducible component of the closure $\ov Z$. By Theorem
\ref{main-wt}\eqref{item:13}, the equations defining $Z$ form a
complete intersection in a suitable neighborhood of~$C$.  This
observation is clear in the case when all the irreducible components
of $V$ are weakly transverse.

\begin{corollary}
  \label{cor:2}
  Let notation and assumptions be as in Theorem \ref{main-wt}.
  Suppose that all the irreducible components of $V$ are weakly
  transverse. Then Algorithm \ref{alg:200} gives every locally closed
  subset $Z\in \Gamma_{\max}$ as a complete intersection outside a
  hypersurface.
\end{corollary}

\begin{proof}
  Let $Z\in \Gamma_{\max}$ and $C$ an irreducible component of $\ov
  Z$. Then $C$ is an irreducible component of $V$, and so it is weakly
  transverse. By Theorem \ref{main-wt}\eqref{item:13}, the equations
  defining $Z$ form a complete intersection in a neighborhood
  of~$C$. Since this holds for all the components of $Z$, it follows
  that this locally closed subset is given as a complete intersection
  outside a hypersurface.
\end{proof}

For instance, if we know {\it a priori} that $\dim(V)=0$ and all
points in $V$ have multiplicatively independent coordinates, then
\begin{displaymath}
V=\bigcup_{Z} Z,
\end{displaymath}
where the union is over the locally closed subsets produced by
Algorithm \ref{alg:200} and given as the zero set of $n$ equations
outside a hypersurface. Each of these locally closed subsets are
either of dimension 0, or the empty set.

\subsection{The general case}\label{sec:general-case}

We devote the rest of this section to the general multivariate
case. We first give a simple subroutine which, given a locally closed
subset $W$ of an algebraic torus and a torsion coset of codimension 1,
computes their intersection as a locally closed subset of an algebraic
torus of lower dimension.

 \begin{algorithm} 
    \caption{(Intersecting with a torsion coset)} \label{alg:1}
    \begin{algorithmic}[1]
\REQUIRE  $N,k\in\N$ such that $N>k$ and a quadruple $(\bfeta,W,T,\delta)$ with $\bfeta\in\upmu_{\infty}^{k}$, $W\subset\Gm^{N-k}$ a locally
closed subset defined over $K(\bfeta)$, $T\subset \Gm^{N-k}$ a torsion
coset of codimension 1, and $\delta\in \Aut(\Gm^N)$.
\ENSURE a triple $(\wt \bfeta,
\wt W,\wt \delta)$ with
$\wt \bfeta\in\upmu_{\infty}^{k+1}$,  $\wt W\subset\Gm^{N-k-1}$ a
locally closed subset defined over $K(\wt \bfeta)$, and  $\wt \delta\in
\Aut(\Gm^N)$.
      \STATE \label{alg:8} Write $T=V(\bfy^{\bfb}-\xi)$
      for a primitive vector
      $\bfb\in \Z^{N-k}$  and $\xi\in \upmu_{\infty}$;
      \STATE \label{alg:21} choose $\tau\in \Aut(\Gm^{N-k})$ such that
      $\tau(T)=V(y_{N-k}-\xi)$ as in Lemma \ref{lemm:1} and
      let $\pi\colon \Gm^{N-k}\to\Gm^{N-k-1}$ be the projection onto
      the first $N-k-1$ coordinates;
      \RETURN \label{alg:2} $\wt \bfeta\leftarrow \xi
      \times\bfeta$, $\wt W\leftarrow \pi(\tau(W)\cap
      V(y_{N-k}-\xi))$ and $\wt \delta\leftarrow (\tau\times \id_{\Gm^{k}}) \circ\delta$;
    \end{algorithmic}
  \end{algorithm}

  \begin{lemma} \label{E1a} For a given input $(\bfeta,W,T,\delta)$,
    the output of Algorithm~\ref{alg:1} satisfies
$$
\delta^{-1}((W\cap T)\times\{\bfeta\})=
\wt \delta^{-1}(\wt W\times\{\wt \bfeta\}).
$$
\end{lemma}

\begin{proof}
  With notation as in the algorithm, we verify that $\tau( W\cap T)=
  \wt W\times\{\xi\}$.  We deduce that
\begin{displaymath}
\delta^{-1}((W\cap T)\times\{\bfeta\})=
(\delta^{-1}\circ (\tau\times \id_{\Gm^{k}})^{-1})(\wt W\times\{\wt \bfeta\})=
\wt \delta^{-1}(\wt W\times\{\wt\bfeta\}),
\end{displaymath}
as stated. 
\end{proof} 

\begin{remark} \label{rem:1} 
The locally closed subset $\wt W$ in Algorithm \ref{alg:1}, line \ref{alg:2}, can be
represented as follows. Write $W=V(\bfF)\setminus V(\bfG)$ with
$F_{j},G_{l}\in \Q[y_{1}^{\pm1},\dots, y_{N-k}^{\pm1}]$. Then
$$
\wt W=V(\wt \bfF)\setminus V(\wt \bfG)
$$ 
with $\wt F_{j}= F_{j}\circ \tau^{-1}(y_{1},\dots, y_{N-k-1},\xi) $
and $\wt G_{l}= G_{l}\circ \tau^{-1}(y_{1},\dots, y_{N-k-1},\xi)$.
\end{remark}

In Algorithm \ref{alg:3} below, we apply the previous procedure to a
 subvariety $W_{0}$ of $\Gm^{N}$ and the torsion cosets successively
produced by Conjecture \ref{Zilber}.  We thus obtain a decomposition
of $W_{0}$ as a union of a finite collection of complete intersections
outside hypersurfaces with empty exceptional subset (Definition
\ref{def:3}).

 \begin{algorithm} 
    \caption{(Descent)} \label{alg:3} 
    \begin{algorithmic}[1]
\REQUIRE  a subvariety $W_{0}\subset \Gm^{N}$ defined over $K$.
\ENSURE for $k=0,\dots, N$, a finite collection $\Lambda_k$ of
triples $(\bfeta,W,\delta)$ with
$\bfeta\in\upmu_{\infty}^k$, $W\subset\Gm^{N-k}$ a locally closed subset
and $\delta\in \Aut(\Gm^N)$.
\STATE \label{alg:9} Set $\Sigma_{0}\leftarrow \{(1_{{\Gm^{0}}}, W_{0},
\id_{\Gm^{N}})\}$, where $1_{\Gm^{0}}$ denotes the neutral element of $\Gm^{0}$;
\FOR{$k$ from 0 to $N$} \label{alg:29}
\STATE set $\Lambda_{k}\leftarrow \emptyset$ and $\Sigma_{k+1}\leftarrow \emptyset$;      
\FOR{$(\bfeta,W,\delta)\in \Sigma_{k}$}
\STATE \label{alg:30} compute a complete intersection stratification
      $(W_{j})_{j}$ of $W$ defined over $K$;
\STATE \label{alg:12} compute the collections $\Omega_{W_{j}}$ for all $j$;
\STATE \label{alg:4} if $W_{j} \not\subset \bigcup\Omega_{W_{j}}$, then add $(\bfeta,W_{j}\setminus
\bigcup\Omega_{W_{j}}, \delta)$ to $\Lambda_{k}$; 
\FOR{each $j$ and $T\in \Omega_{W_{j}}$} \label{alg:52}
\STATE \label{alg:10} apply Algorithm \ref{alg:1} to $(\bfeta,W_{j},
T, \delta)$ and add its output to $\Sigma_{k+1}$;
\ENDFOR  \label{alg:53}
\ENDFOR  
\ENDFOR \label{alg:34}
    \end{algorithmic}
  \end{algorithm}

\begin{lemma}
\label{descent}
Assume that Conjecture \ref{Zilber} holds. Given a subvariety
$W_0\subset\Gm^N$ defined over $K$, Algorithm \ref{alg:3} stops after
a finite number of steps and its output has the following
properties:
\begin{enumerate}
\item \label{item:9}
$$
W_0=\bigcup_{k=0}^{N}\bigcup_{(W,\bfeta,\delta)\in\Lambda_k} \delta^{-1}(W\times\{\bfeta\});
$$
\item \label{item:1} the locally closed subset $W$ in a triple
  $(\bfeta,W,\delta)\in \Lambda_{k}$ is given by a collection of
  Laurent polynomials over $K(\bfeta)$ defining a complete
  intersection in the complement of a hypersurface;
\item \label{item:10} if $(\bfeta,W,\delta)\in \Lambda_{k}$ for some
  $k$ and $Z\subset \Gm^{N}$ is an irreducible locally closed subset
  contained in $\delta^{-1}(W\times\{\bfeta\})$, then
\begin{displaymath}
\rank(Z)-\dim(Z)\ge  \codim(W).
\end{displaymath}
\end{enumerate}
\end{lemma}

\begin{proof} 
  If Conjecture \ref{Zilber} holds, the computation of the collections
  $\Omega_{W_{j}}$ in line \ref{alg:12} of the algorithm can be performed
  and so Algorithm \ref{alg:3} stops after a finite number of steps.

We first  show that, for $k=0,\dots, N+1$,  
\begin{equation} \label{eq:7}
W_{0}= 
\bigg(\bigcup_{l=0}^{k-1}\bigcup_{(\bfeta,W,\delta)\in 
\Lambda_{l}} \delta^{-1}(W\times\{\bfeta\}) \bigg)\cup \bigg(\bigcup_{(\bfeta,W,\delta)\in \Sigma_{k}}
\delta^{-1}(W\times\{\bfeta\})\bigg).
\end{equation}

This is clear for $k=0$, because  of the definition of $\Sigma_{0}$ in
line \ref{alg:9} and the fact that the first union in the right-hand
side is empty. For $k\ge1$, 
the construction of the collections $\Lambda_{k}$
and $\Sigma_{k+1}$ in lines \ref{alg:4} and \ref{alg:10} together with Lemma
\ref{E1a} implies that
\begin{displaymath}
\bigcup_{(\bfeta,W,\delta)\in \Sigma_{k}}
\delta^{-1}(W\times\{\bfeta\})=
\bigg(\bigcup_{(\bfeta,W,\delta)\in 
\Lambda_{k}} \delta^{-1}(W\times\{\bfeta\}) \bigg)\cup \bigg(\bigcup_{(\bfeta,W,\delta)\in \Sigma_{k+1}}
\delta^{-1}(W\times\{\bfeta\})\bigg).
\end{displaymath}
Then \eqref{eq:7} follows from the
inductive hypothesis. For $(\bfeta,W,\delta)\in \Sigma_{N}$, the
collections $\Omega_{W_{j}}$ are empty. Hence,  $\Sigma_{N+1}=\emptyset$. 
The statement \eqref{item:9} then follows from the case
$k=N+1$ of \eqref{eq:7}.

Statement \eqref{item:1} is clear from the construction of $\Lambda_{k}$ in
line  \ref{alg:4}.

To prove \eqref{item:10}, let $Z'$ be the locally closed subset of $W$
such that $Z=\delta^{-1}(Z'\times\{\bfeta\})$. The locally closed
subset $W$ is equidimensional and has empty exceptional subset, since
it is the complement in~$W_{j}$ of the collection
$\Omega_{W_{j}}$. Hence $Z'$ is not atypical (Definition~\ref{def:3}),
which implies that
\begin{equation*}
  \rank(Z')-\dim(Z')=  \codim(W).
\end{equation*}
In turn, this implies \eqref{item:10} since  $\rank(Z)=\rank(Z')$ and
$\dim(Z)=\dim(Z')$.
\end{proof}

\begin{remark} \label{rem:5} The condition $W_{j}\not\subset\bigcup
  \Omega_{W_{j}}$ in line \ref{alg:4} is equivalent to the fact that
  $W_{j}$ has at least one irreducible component that is weakly
  transverse, see Remark \ref{rem:4}.  This test avoids
  adding to the collection $\Lambda_{k}$ a triple with an empty
  locally closed set.
\end{remark}

Algorithm \ref{alg:5} below gives the procedure for the reduction of
overdetermined systems. First, it applies Algorithm \ref{alg:3} to
decompose the linear subvariety $W\subset\Gm^{N}$ into pieces without
exceptional subset. Then, it produces the sought decomposition of the
zero set of the given system of equations as the inverse image of
these pieces with respect to the homomorphism $\varphi$.

 \begin{algorithm}     \caption{(Reduction of overdetermined systems)} \label{alg:5}
    \begin{algorithmic}[1]
\REQUIRE  a subvariety $V\subset\Gm^{n}$ defined over $K$.
\ENSURE a finite collection $\Gamma$ of locally closed subsets $Y\subset\Gm^{n}$ 
defined over a cyclotomic extension of $K$.
\STATE  Let $f_i$,
 $i=1,\dots, s$, be the Laurent polynomials 
   defining~$V$ and write  $f_i=\sum_{j=1}^N
   \alpha_{i,j}\bfx^{\bfa_j}$ with $\alpha_{i,j}\in K$ and
   $\bfa_{j}\in \bigcup_{i}\supp(f_{i})$;
\STATE \label{alg:101} set $W\leftarrow V(\sum_{j=1}^N \alpha_{1,j}y_j, \dots, \sum_{j=1}^N \alpha_{s,j}y_j)\subset\Gm^N$;
\STATE  \label{alg:102} apply Algorithm \ref{alg:3} to $W$ and set
$(\Lambda_{k})_{0\le k\le N}$
for its output;
\STATE let $\varphi\in \Hom(\Gm^{n}, \Gm^{N})$ be the homomorphism
corresponding to the exponents $\bfa_{j}\in\Z^{n}$, $j=1,\dots, N$;
\FOR{$k$ from 0 to $N$}
\STATE  set $\pi_1\colon\Gm^N\rightarrow\Gm^{N-k}$
and $\pi_2\colon\Gm^N\rightarrow\Gm^k$ for the projections onto the first
$N-k$ coordinates and  the last $k$ coordinates,
respectively;
\FOR{$(\bfeta,W,\delta)\in \Lambda_{k}$}  \label{alg:6}
\STATE \label{alg:7} apply Algorithm \ref{alg:11} to the homomorphism
$\pi_2\circ\delta\circ\varphi\colon \Gm^{n}\to \Gm^{k}$ and the torsion
coset~$\{\bfeta\}\subset \Gm^{k}$;
\IF{$\varphi^{-1}(\bfeta)\ne \emptyset$}
\STATE \label{alg:13} let $\bfE=\{p_{j}\}_{1\le j\le t}$ be the output
of Algorithm \ref{alg:11};
\STATE let $F_{j}$, $j=1,\dots, l$, and $G$ be the Laurent
polynomials over $K(\bfeta)$ defining~$W$;
\STATE \label{alg:31} set $p_{t+j}\leftarrow
F_{j}\circ\pi_1\circ\delta\circ\varphi$, $j=1,\dots, l$, and 
$q\leftarrow G\circ\pi_1\circ\delta\circ\varphi$;
\STATE set $Y\leftarrow V(p_{1},\dots, p_{t+l})\setminus V(q)$ and add
$Y$ to $\Gamma$;
\ENDIF
\ENDFOR
\ENDFOR
    \end{algorithmic}
  \end{algorithm}

\begin{theorem}
\label{main}
Assume that Conjecture \ref{Zilber} holds. Given a subvariety
$V\subset\Gm^n$ defined over $K$, Algorithm \ref{alg:5} 
stops after a finite number of steps and its output satisfies:
\begin{enumerate}
\item \label{item:11} 
\begin{equation*}
V=\bigcup_{Y\in \Gamma} Y;
\end{equation*}
\item \label{item:2} each locally closed subset $Y\in \Gamma$ is
  either given as a complete intersection in the complement of a
  hypersurface, or it is the empty set.
\end{enumerate}

If $n$ is bounded and $V$ is defined by a bounded number of Laurent
polynomials of degree $\leq d$, bounded height and number of nonzero
coefficients, then the cardinality of $\Gamma$ is bounded by $O(1)$,
the Laurent polynomials defining each $Y\in \Gamma$ are defined over a
cyclotomic extension of degree $O(1)$, have degree bounded by
$d^{O(1)}$, and height and number of nonzero coefficients bounded by
$O(1)$. The complexity of the algorithm is bounded by
$O(\MM(\log(d))\log(\log(d)))= \wt O(\log (d))$~ops.
\end{theorem}

\begin{proof}
  Let $0\le k\le N$ and $(\bfeta,W,\delta)\in \Lambda_{k}$ be as in line
  \ref{alg:6} of the algorithm. In case $\varphi^{-1}(\bfeta)\ne
  \emptyset$, we denote by $Y$ the locally closed
  subset associated to this triple. By construction,
  \begin{equation}
    \label{eq:4}
    (\delta\circ \varphi)^{-1}(W\times\{\bfeta\})=
\begin{cases}
\emptyset& \text{ if } \varphi^{-1}(\bfeta)=
  \emptyset, \\
Y &\text{ if } \varphi^{-1}(\bfeta)\ne
  \emptyset.
\end{cases}
  \end{equation}
The decomposition in \eqref{item:11} then follows from the
one in Lemma \ref{descent}\eqref{item:9} and the fact that
  $\varphi^{-1}(W)=V$.

  To prove \eqref{item:2}, suppose that $Y$ is nonempty and let $C$ be
  one of its irreducible components. By~\eqref{eq:4}, there is an
  irreducible component $Z$ of $(W\times\{\bfeta\})\cap
  \im(\delta\circ\varphi)$ such that
  $C=(\delta\circ\varphi)^{-1}(Z)$. Applying Lemmas \ref{lemm:2} and
  \ref{descent}\eqref{item:10}, we deduce that
\begin{equation}\label{eq:5}
  \rank(C)-\dim(C) = \rank(Z)-\dim(Z)\ge \codim(W)=l.
\end{equation}
On the other hand, we have that $C\subset
(\delta\circ\varphi)^{-1}(\bfeta)$ and so $\rank(C)\le
n-t$, thanks to Lemma \ref{torsion-coset}. Together with \eqref{eq:5}, this implies that
\begin{displaymath}
\codim(C)= n-\dim(C) \ge t+l.   
\end{displaymath}
Since $Y$ is defined by the $t+l$ Laurent polynomials $p_{1},\dots, p_{t+l}$
outside the hypersurface $V(q)$, it follows that
$\codim(C)=t+l$. Hence, $Y$ is a complete intersection outside $V(q)$,
as stated.

Now assume that both $n$ and $s$ are bounded and that each $f_{i}$ is
of degree $\leq d$, of bounded height and number of nonzero
coefficients.  The variety $W\subset \Gm^{N}$ in line \ref{alg:101}
does not depend on $d$.  Hence, the application of Algorithm
\ref{alg:3} in line \ref{alg:102} produces a output of size $O(1) $
using $O(1)$ ops.  In particular, the collection $\Gamma$ has
cardinality bounded by $O(1)$.  
Lemma \ref{torsion-coset} shows that the
binomials in the collection $\bfE$ in line \ref{alg:13} have
coefficients in $\upmu_{l}$ for $l\le O(1)$, that the Laurent
polynomials defining $Y$ have degree, height and number of nonzero
coefficients as predicted by Theorem~\ref{main}, and that the
complexity of this step is bounded by
$O(\MM(\log(d))\log(\log(d)))$~ops.  From this, we deduce that the
overall complexity is bounded by
$O(\MM(\log(d))\log(\log(d)))=\wt O(\log(d))$ ops.
\end{proof}

\begin{remark} \label{rem:7} For $Y\in \Gamma$, the defining equations
  and inequations come from different sources. In the notation of
  Algorithm \ref{alg:5}, the Laurent polynomials $p_{i}$, $i=1,\dots,
  t$, are binomials with coefficients in $\upmu_{\infty}$, whereas
  $p_{i}$, $i=t+1,\dots, t+l$ come from the equations defining $W$. If
  $Y\ne\emptyset$ and $C$ is an irreducible component of $Y$, we have
  that
\begin{displaymath}
t=n-  \rank(C), \quad l=\rank(C)-\dim(C).
\end{displaymath}
\end{remark}

Theorem \ref{main-s} in the introduction follows by decomposing the
affine space $\A^{n}$ as a disjoint union of tori and applying Theorem
\ref{main} to each of them.

\begin{proof}[Proof of Theorem \ref{main-s}]
Given a subset $I\subset \{1,\dots, n\}$, we consider the
  locally closed subset
  $G_{I}=\{\bfx\in \A^{n}\mid x_{i}\ne 0 \text{ if and only if } i\in
  I\}$. The affine space then decomposes as a disjoint union 
\begin{displaymath}
  \A^{n}=\bigsqcup_{I} G_{I},
\end{displaymath}
and each $G_{I}$ is an algebraic torus $\Gm^{\# I}$ embedded into the
standard linear subspace $V(\{x_{i}\mid i\notin I\})$ of $\A^{n}$.

Given a system of equations over $\A^{n}$, we can split it into
$2^{n}$ systems of equations over these algebraic tori. For each $I$,
we solve the corresponding system of equations by applying Algorithm
\ref{alg:5} and we multiply the obtained Laurent polynomials by
suitable monomials in order to clear all possible
denominators. Finally, we add the set of variables $x_{i}$, $i\notin
I$, to the obtained equations, and we multiply the polynomial defining
the open subset by the monomial $\prod_{i\in I}x_{i}$.

By Theorem \ref{main}\eqref{item:11}, the resulting polynomials form a
collection of systems of equations which define either a complete
intersection in the complement of a hypersurface, or the empty set. By
Theorem \ref{main}\eqref{item:2}, this collection gives a
decomposition of $V(f_{1},\dots, f_{s})$ as in \eqref{eq:8}. The rest
of the statement follows also from Theorem \ref{main}.
\end{proof}

\begin{remark} \label{rem:3} In practice, there are a number of
  modifications that can be applied to our general procedure. They do
  not affect the theoretical complexity of the algorithms, but can
  significantly simplify the computations in concrete examples.

\begin{enumerate}
\item \label{orto} For a given system of equations, it is better to
  apply Algorithm~\ref{alg:33} several times, starting with a subtorus
  $T\subset\Gm^{N}$ of codimension 1 of small degree and the linear
  subvariety $W\subset \Gm^{N}$. If it turns that $\im(\varphi)\subset
  T$, then this procedure reduces the dimension of the ambient space
  without breaking $W$ into several pieces, as might happen when
  applying the descent in Algorithm~\ref{alg:5}.

\item \label{item:15} Both in Algorithms \ref{alg:200} and
  \ref{alg:5}, one can replace the linear subvariety  $W\subset
  \Gm^{N}$ by the subvariety of $\Gm^{N-1}$ given by
\begin{displaymath}
   V\bigg(\alpha_{1,1}+ \sum_{j=2}^N \alpha_{1,j}y_j, \dots, \alpha_{s,1}+ \sum_{j=2}^N \alpha_{s,j}y_j\bigg)
\end{displaymath}
and $\varphi$ by the homomorphism associated to the vectors
$\bfa_{j}-\bfa_{1}$, $j=2,\dots, N$. In this way, computations start
in a space of dimension $N-1$ instead of one of dimension $N$.

\item \label{item:16} The locally closed subsets $Y$ of $\Gm^{n}$
  produced by Algorithm \ref{alg:5} have codimension bounded by $n$.
  Hence, in line \ref{alg:6} of this algorithm, it suffices to consider
  only the triples $(\bfeta,W,\delta)$ such that $\codim(W)\le n$.
  Consequently, in line \ref{alg:30} of Algorithm \ref{alg:3}, it
  suffices to compute only the components $W_{j}$ in the complete
  intersection stratification of $W$ of codimension bounded by $n$.  A
  similar remark  applies to Algorithm \ref{alg:200} for the weakly
  transverse case. 
\end{enumerate}
\end{remark}

\section{Examples} \label{examples}

We illustrate with two examples how our algorithms work.  We will
systematically use the modifications in Remark \ref{rem:3}. To shorten
the presentation, we will only compute the zeros of these systems in
the algebraic torus.

\begin{example} \label{exm:1}
Let   $d\ge1$ and consider the system of polynomials
\begin{equation}\label{eq:28}
f_1=x_{1}^dx_{2}^2-5x_{1}^dx_{2}-2x_{2}+10,\quad 
f_2=x_{1}^{d+1}x_{2}-2x_{1}^dx_{2}-2x_{1}+4 \in\Q[x_{1},x_{2}].
\end{equation}
Its zero set in $\Gm^{2}$ consists of the curve defined by the
polynomial $x_{1}^dx_{2}-2$ (which is a common factor of $f_1$ and
$f_2$) and the isolated point $(2,5)$. In the sequel, we describe how
our algorithms give this result.

The support of these polynomials consists of  the vectors
$(d+1,1), (d,2), (d,1),$ $ (0,1), (1,0), (0,0) \in\Z^{2}$.   
Let $W \subset\Gm^5$ be  the subvariety  defined by
$$
F_1=y_2-5y_3-2y_4+10, \quad F_2=y_1-2y_3-2y_5+4\in \Q[y_{1}^{\pm1},
\dots, y_{5}^{\pm1}]
$$
and $\varphi\colon\Gm^2\rightarrow\Gm^5$ the homomorphism given by
\begin{displaymath}
\varphi(x_{1},x_{2})=(x_{1}^{d+1}x_{2},x_{1}^dx_{2}^2,x_{1}^dx_{2},x_{2},x_{1}),   
\end{displaymath}
so that
\begin{math}
V(f_{1},f_{2})=\varphi^{-1}(V(F_{1},F_{2})).
\end{math}

The subtorus $T=V(y_2y_3^{-1}y_4^{-1}-1)$ satisfies $\im(\varphi)
\subset T$. Following Remark \ref{rem:3}\eqref{orto}, we apply
Algorithm \ref{alg:33} at this point, instead of the descent procedure
in Algorithm~\ref{alg:3}.  We choose the automorphism $\tau\in
\Aut(\Gm^{5})$ given by
$$
\tau(y_1,y_2,y_3,y_4,y_5)=(y_1,y_3,y_4,y_5,y_2y_3^{-1}y_4^{-1}).
$$ 
It satisfies $\tau(T)=V(y_5-1)$. The corresponding output of Algorithm
\ref{alg:33} for the triple $(W,T,\varphi)$ is the subvariety $\wt
W\subset \Gm^{4}$ defined by
\begin{align*}
\wt F_{1}&=F_1(y_1,y_2y_3,y_2,y_3,y_4)=
y_2y_3-5y_2-2y_3+10,  \\
\wt F_{2}&=F_2(y_1,y_2y_3,y_2,y_3,y_4)=y_1-2y_2-2y_4+4, 
\end{align*}
and the homomorphism $\wt\varphi\colon \Gm^{2}\to \Gm^{4}$ given by
$\wt\varphi(x_{1},x_{2})=(x_{1}^{d+1}x_{2},x_{1}^dx_{2},x_{2},x_{1})$. 

Set $(W,\varphi)\leftarrow (\wt W,\wt \varphi)$. We apply again
Algorithm \ref{alg:33}, this time to the triple $(W,T,\varphi)$ with
$T=V(y_1y_2^{-1}y_4^{-1}-1)$. This subtorus satisfies
$\im(\varphi)\subset T$. We choose $\tau\in \Aut(\Gm^{4})$ given by
$\tau(y_1,y_2,y_3,y_4)=(y_2,y_3,y_4,y_1y_2^{-1}y_4^{-1})$, which
satisfies $\tau(T)=V(y_4-1)$.  The corresponding output of Algorithm
\ref{alg:33} is the subvariety $\wt W\subset \Gm^{3}$ defined by
\begin{align*}
\wt F_{1}&= F_1(y_1y_3,y_1,y_2,y_3)=y_1y_2-5y_1-2y_2+10 ,\\ 
\wt F_{2}&= F_2(y_1y_3,y_1,y_2,y_3) =y_1y_3-2y_1-2y_3+4,
\end{align*}
and the homomorphism $\wt \varphi\colon \Gm^{2}\to \Gm^{3}$ given by
$\wt \varphi(x_{1},x_{2})=(x_{1}^dx_{2},x_{2},x_{1})$.

Set again $(W,\varphi)\leftarrow (\wt W,\wt \varphi)$.  There is
no proper subtorus of $\Gm^{3}$ of degree independent of $d$ and
containing the image of $\varphi$. Hence, we cannot further apply
Remark~\ref{rem:3}\eqref{orto}, at least when $d\gg 0$.  Instead, we
apply the  procedure in Algorithm~\ref{alg:5}.

The subvariety $W$ has two irreducible
components, of codimension $1$ and~$2$, respectively. Indeed,
$$
W=V(y_1-2)\cup V(y_2-5,y_3-2).
$$ 
Following line \ref{alg:102} of Algorithm \ref{alg:5}, we apply
Algorithm \ref{alg:3} to this subvariety in order to construct the
collections $\Lambda_{k}$. As in line \ref{alg:9} of this algorithm,
we set $\Sigma_{0}=\{(1_{\Gm^{0}}, W, \id_{\Gm^{3}})\}$, where
$1_{\Gm^{0}}$ denotes the neutral element of the trivial group
$\Gm^{0}$.  We now describe what is done in the loop between lines
\ref{alg:29} and \ref{alg:34}.

Set $k=0$. There is only one element in $\Sigma_{0}$, namely
$(1_{\Gm^{0}}, W, \id_{\Gm^{3}})$. A complete intersection
stratification of $W$ is given by
\begin{displaymath}
W_1=V(y_1-2)\backslash V(y_3-2), \quad
W_2=V((y_1-2)(y_2-5),y_3-2).
\end{displaymath}
It is easy to check that $W_1$ has empty exceptional subset and that
the only maximal atypical subvariety of $W_{2}$ is given by $W_2\cap
T=V(y_{1}-2,y_{3}-2)$ for the subtorus $T=V(y_1y_3^{-1}-1)$. Hence, we
may choose $\Omega_{W_{1}}=\emptyset$ and $\Omega_{W_2}=\{T\}$ in line
\ref{alg:12} and we add to $\Lambda_{0}$ the triples $(1_{\Gm^{0}},
W_{1}, \id_{\Gm^{3}})$ and $(1_{\Gm^{0}}, W_{2}\setminus T,
1_{\Gm^{3}})$ in line \ref{alg:4}. Let us consider now the loop
between lines \ref{alg:52} and \ref{alg:53}. Since $\Omega_{W_2}$
consists of the only torsion coset $\{T\}$, we apply Algorithm
\ref{alg:1} to the quadruple $(1_{\Gm^{0}}, W_{2}, T,
\id_{\Gm^{3}})$. We choose $\tau\in \Aut(\Gm^{3})$ given by
$\tau(y_{1},y_{2},y_{3})=(y_1,y_2,y_1^{-1}y_3)$. Thus
$\tau(T)=V(y_3-1)$ as required, and
$$
\tau(W_2\cap T)=\wt W\times\{1\}, 
$$ 
with $\wt W= V(y_1-2)\subset\Gm^2$. Hence, Algorithm \ref{alg:1} gives
$(1, \wt W,\tau)$ as output. We add this element to $\Sigma_{1}$ in
line \ref{alg:10}. 

Set now $k=1$. The only element of $\Sigma_{1}$ is $(1, \wt W,\tau)$
and $\wt W=V(y_1-2)$ has no atypical locally closed subset. Thus, we
 choose $\Omega_{\wt W}=\emptyset$ and  add $(1, \wt W,\tau)$ to
$\Lambda_{1}$ in line \ref{alg:4}. The construction of Algorithm
\ref{alg:3} ends up here.

We now construct the collection $\Gamma$ in Algorithm \ref{alg:5}.  Recall
that $\varphi\colon \Gm^{2}\to\Gm^{3}$ is the homomorphism given by
$$
\varphi(x_{1},x_{2})=(x_{1}^dx_{2},x_{2},x_{1}).
$$  
The elements of
$\Lambda_0$ are $(1_{\Gm^{0}}, W_{1}, \id_{\Gm^{3}})$ and
$(1_{\Gm^{0}}, W_{2}\setminus T, \id_{\Gm^{3}})$.  
These triples 
contribute to $\Gamma$ with the locally closed subsets
\begin{align*}
Y_{1}&=V(x_{1}^dx_{2}-2)\backslash V(x_{1}-2),\\
Y_{2}&=V((x_{1}^dx_{2}-2)(x_{2}-5),x_{1}-2)\backslash
V(x_{1}^{d-1}x_{2}-1)=\{(2,5)\}. 
\end{align*}  
The only triple in $\Lambda_1$ is $(1, \wt W,\tau)$, and it
contributes to $\Gamma$ with the locally closed subset
\begin{displaymath}
  Y_{3}=V(x_{1}^dx_{2}-2,x_{1}^{d-1}x_{2}-1)=\{(2,2^{1-d})\}.
\end{displaymath}
By Theorem~\ref{main}, the zero set of the system \eqref{eq:28}
decomposes as $Y_{1}\cup Y_{2} \cup Y_{3}$.
\end{example}

\begin{example}
  \label{exm:2}
Let $d\ge1$ and consider the system of polynomials 
\begin{multline}
  \label{eq:29}
f_1=x_{1}^{3d+1}x_{2}^{3d}+x_{1}^2x_{2}+5,
f_2=x_{1}^{3d+2}x_{2}^{3d}+5x_{1}+25, \\
f_3=x_{1}+x_{1}^2x_{2}+25x_{2} \in \Q[x_{1},x_{2}].  
\end{multline}
Its zero set in $\Gm^{2}$ consists of two points $(5\zeta,\zeta/5)$ of
rank $1$, with $\zeta$ a primitive third root of unity. As in the
previous example, we will describe how our algorithms give this
result.

The support of $f_{1}$ and $f_{2}$ consists of  the vectors
$ (0,0),
({3d+1},{3d}), ({3d+2}, {3d})$, $(2,1), (1,0), (0,1)\in \Z^{2}$.   
Let $W\subset \Gm^{5}$ be the subvariety defined by 
$$
F_1=y_1+y_3+5,\quad F_2=y_2+5y_4+25,\quad F_3=y_4+y_3+25y_5
$$
and $\varphi\colon\Gm^2\rightarrow\Gm^5$  the homomorphism given by
$$
\varphi(x_{1},x_{2})=(x_{1}^{3d+1}x_{2}^{3d},x_{1}^{3d+2}x_{2}^{3d},x_{1}^2x_{2},x_{1},x_{2}).
$$

The subtorus $T=V(y_1y_2^{-1}y_4-1)$ satisfies $\im(\varphi) \subset
T$. Following Remark \ref{rem:3}\eqref{orto}, we apply Algorithm
\ref{alg:33} to the triple $(W,T,\varphi)$. We choose $\tau\in
\Aut(\Gm^{5})$ given by
$\tau(y_1,y_2,y_3,y_4,y_5)=(y_2,y_1y_2^{-1},y_3,y_5,y_1y_2^{-1}y_4)$
as one of the automorphisms that satisfies $\tau(T)=V(y_5-1)$. The corresponding
output of this algorithm is the subvariety $\wt W\subset\Gm^{4}$
defined by $\wt F_{i}=F_i(y_1y_2,y_1,y_3,y_2^{-1},y_4)$, $i=1,2,3$,
that is
\begin{equation*}
\wt F_{1}= y_1y_2+y_3+5,\quad 
\wt F_{2}=y_1+5y_2^{-1}+25, \quad
\wt F_{3}=y_2^{-1}+y_3+25y_4,
\end{equation*}
and the homomorphism $\wt \varphi\colon \Gm^{2}\to \Gm^{4}$ given by
$\wt \varphi(x_{1},x_{2})=
(x_{1}^{3d+2}x_{2}^{3d},x_{1}^{-1},x_{1}^2x_{2},x_{2})$.

Set $(W,\varphi)\leftarrow (\wt W,\wt \varphi)$. We apply again
Algorithm \ref{alg:33}, this time to the triple $(W,T,\varphi)$ with
$T=V(y_2^2 y_3y_4^{-1}-1)$. This subtorus satisfies
$\im(\varphi)\subset T$. We choose $\tau\in\Aut(\Gm^{4})$ given by
$\tau(y_1,y_2,y_3,y_4)=(y_1,y_2^{-2}y_4,y_2,y_2^2 y_3y_4^{-1})$, which
satisfies $\tau(T)=V(y_4-1)$.
The corresponding output of
Algorithm \ref{alg:33} is subvariety
$\wt W\subset \Gm^{3}$ defined by $\wt F_{i} =
F_i(y_1,y_3,y_2,y_2y_3^2)$, $i=1,2,3$, that is
\begin{equation*}
\wt F_{1} =y_1y_3+y_2+5,\quad
\wt F_{2} =y_1+5y_3^{-1}+25,\quad
\wt F_{3} =y_3^{-1}+y_2+25y_2y_3^2, 
\end{equation*}
and the homomorphism $\wt \varphi\colon \Gm^{2}\to\Gm^{3}$ given by
$\wt \varphi(x_{1},x_{2})=(x_{1}^{3d+2}x_{2}^{3d},x_{1}^2x_{2},x_{1}^{-1})$.

Set again $(W,\varphi)\leftarrow (\wt W,\wt \varphi)$.  We apply for a
third time Algorithm \ref{alg:33}, this time to the triple
$(W,T,\varphi)$ with $T=V(y_1y_3^2-1)$. This subtorus satisfies
$\im(\varphi)\subset T$. We choose $\tau\in \Aut(\Gm^{3})$ given by
$\tau(y_1,y_2,y_3)=(y_1,y_3,y_1y_3^{2})$, which satisfies
$\tau(T)=V(y_3-1)$.   The subvariety $\wt W\subset \Gm^{2}$ 
in the output of Algorithm
\ref{alg:33} is defined by
\begin{equation*}
\wt F_{1}=F_1(y_2^{-2},y_1,y_2)=y_2^{-1}+y_1+5,\quad \wt F_{2}= F_2(y_2^{-2},y_1,y_2) = y_2^{-2}+5y_2^{-1}+25,
\end{equation*}
because  the Laurent polynomial $\wt F_{3}=
F_3(y_2^{-2},y_1,y_2) =y_2^{-1}+y_1+25y_1y_2^2$  
lies in the ideal $(\wt F_{1},\wt F_{2})$. 
The corresponding homomorphism $\wt \varphi\colon \Gm^{2}\to \Gm^{2}$
is given by $\wt \varphi(x_{1},x_{2})=(x_{1}^2x_{2},x_{1}^{-1},x_{1}^{3d}x_{2}^{3d})$.

Again set $(W,\varphi)\leftarrow (\wt W,\wt \varphi)$. There exists no
proper subtorus of $\Gm^{2}$ of degree independent of $d$ such that
$\im(\varphi)\subset T$. Thus we cannot further apply the
Algorithm~\ref{alg:33} when $d\gg 0$. Instead, we  apply the general
procedure in Algorithm
\ref{alg:5}.  As indicated in Algorithm \ref{alg:5}, line
\ref{alg:102}, we apply Algorithm \ref{alg:3} to the subvariety
$W\subset\Gm^{2}$ to construct the sets $\Lambda_k$.  We describe what
is done in the loop in lines \ref{alg:29} to \ref{alg:34} of this
algorithm.

Set $k=0$ and choose at line \ref{alg:9} the only element of
$\Sigma_{0}$, namely $(1_{\Gm^{0}}, W, \id_{\Gm^{2}})$.  The
subvariety $W$ consists of the two points $(5\zeta,\zeta/5)$ with
$\zeta$ a primitive third root of unity. In particular, it is
0-dimensional and is already given as a complete intersection.
Moreover, these two points are atypical since
\begin{displaymath}
  W\cap T_{\zeta}= \{(5\zeta,\zeta/5)\}
\end{displaymath}
where $T_\zeta$ is the torsion coset $V(y_1y_2-\zeta^2)$ for $\zeta\in
\upmu_{3}\setminus \{1\}$. Hence we may choose
\begin{math}
  \Omega_{W}=\{T_{\zeta}\}_{\zeta\in \upmu_{3}\setminus \{1\}}
\end{math} in line \ref{alg:12}.
We have that $W\subset \OmegaW$ and so $\Lambda_{0}=\emptyset$. 

Now fix $\zeta\in \upmu_{3}\setminus \{1\}$ and apply Algorithm
\ref{alg:1} to the quadruple $(1_{\Gm^{0}}, W, T_{\zeta},
\id_{\Gm^{2}})$.  We choose $\tau\in \Aut(\Gm^{2})$ given by
$\tau(y_1,y_2)=(y_1,y_1y_2)$. This automorphism satisfies
$\tau(T_\zeta)=V(y_2-\zeta^2)$ as required, and
$$
\tau(W\cap T_\zeta)=\wt W_\zeta\times\{\zeta^2\}
$$ 
with $\wt W_\zeta=V(y_1-5\zeta)=\{5\zeta\}\subset\Gm$.  The
 output of Algorithm \ref{alg:1} is the triple
$(\zeta^{2},\wt W_{\zeta},\tau)$, which we add to $\Sigma_{1}$. 
Hence, at the end of the loop between  lines
\ref{alg:52} and \ref{alg:53},
\begin{displaymath}
  \Sigma_{1}=\{(\zeta^{2},\wt W_{\zeta},\tau)\}_{\zeta\in \upmu_{3}\setminus \{1\}}.
\end{displaymath}
Set now $k=1$ in the loop between lines \ref{alg:29} and
\ref{alg:34}. The subvariety $\wt W_{\zeta}$ has no atypical component
because it is 0-dimensional and contains no torsion point. Thus we
choose $\Omega_{W_{\zeta}}=\emptyset$ and add to $\Lambda_{1}$ the two
elements $(\zeta^{2},\wt W_{\zeta},\tau)$ for $\zeta\in
\upmu_{3}\setminus \{1\}$. The construction of Algorithm \ref{alg:3}
finishes here.

We now construct the collection $\Gamma$ in Algorithm \ref{alg:5}. Recall that
$\varphi\colon \Gm^{2}\to \Gm^{2}$ is given by
\begin{displaymath}
 \varphi(x_{1},x_{2})=(x_{1}^2x_{2},x_{1}^{-1},x_{1}^{3d}x_{2}^{3d}).   
\end{displaymath}
The collection $\Lambda_{1}$ contributes to $\Gamma$ with the locally
closed subsets 
\begin{displaymath}
  Y_{\zeta}=V(x_{1}^2x_{2}-5\zeta, x_{1}x_{2}-\zeta^2)\setminus V(1) =
\Big  \{\Big(5\zeta,\frac{\zeta}{5}\Big)\Big\} \quad \text{ for }
\zeta\in
  \upmu_{3}\setminus \{1\}.
\end{displaymath}
By Theorem~\ref{main},  the zero set of the system \eqref{eq:29}
decomposes as the union of these two points. 
\end{example}


\newcommand{\noopsort}[1]{} \newcommand{\printfirst}[2]{#1}
  \newcommand{\singleletter}[1]{#1} \newcommand{\switchargs}[2]{#2#1}
  \def\cprime{$'$}
\providecommand{\bysame}{\leavevmode\hbox to3em{\hrulefill}\thinspace}
\providecommand{\MR}{\relax\ifhmode\unskip\space\fi MR }
\providecommand{\MRhref}[2]{%
  \href{http://www.ams.org/mathscinet-getitem?mr=#1}{#2}
}
\providecommand{\href}[2]{#2}

\end{document}